\documentclass[12pt,a4paper]{article}

\usepackage{amsmath,amstext,amssymb,amscd,euscript}
 \usepackage{graphicx}

\usepackage[russian,english]{babel}
\usepackage[cp1251]{inputenc}

\newcommand{\Xcomment}[1]{}
\newtheorem{theorem}{Theorem}[section]
\newtheorem{lemma}[theorem]{Lemma}
\newtheorem{corollary}[theorem]{Corollary}
\newtheorem{prop}[theorem]{Proposition}

\newenvironment{proof}{\noindent{\bf Proof}\/}%
{\hfill$\qed$\medskip}

\def\qed{\Box}

\makeatletter \@addtoreset{equation}{section} \makeatother

\newenvironment{numitem1}{\refstepcounter{equation}\begin{enumerate}%
\item[(\thesection.\arabic{equation})]}{\end{enumerate}}

\newcommand{\refeq}[1]{(\ref{eq:#1})}  % reference to equation

 \makeatletter
\renewcommand{\section}{\@startsection{section}{1}{0pt}%
{-3.5ex plus -1ex minus -.2ex}{2.3ex plus .2ex}%
{\normalfont\Large}}
 \makeatother

 \makeatletter
\renewcommand{\subsection}{\@startsection{subsection}{2}{0pt}%
{-3.0ex plus -1ex minus -.2ex}{1.5ex plus .2ex}%
{\normalfont\normalsize\bf}}
 \makeatother

 \newcommand{\SEC}[1]{\ref{sec:#1}}  % reference to section
\newcommand{\SSEC}[1]{\ref{ssec:#1}}  % reference to subsection
\newcommand{\SSSEC}[1]{\ref{sssec:#1}}  % reference to subsubsection
\def\Rset{{\mathbb R}}

\def\Zset{{\mathbb Z}}

\def\Ascr{{\cal A}}
\def\Bscr{{\cal B}}

\def\Dscr{{\cal D}}
\def\Escr{{\cal E}}

\def\Gscr{{\cal G}}

\def\Iscr{{\cal I}}

\def\Lscr{{\cal L}}

\def\Pscr{{\cal P}}

\def\Rscr{{\cal R}}
\def\Sscr{{\cal S}}
\def\Tscr{{\cal T}}

\def\Xscr{{\cal X}}

\def\tilde{\widetilde}
\def\hat{\widehat}

\def\eps{\varepsilon}

\def\deltain{\delta^{\rm in}}
\def\deltaout{\delta^{\rm out}}

\def\xmin{x^{\rm min}}
\def\xmax{x^{\rm max}}

\def\bmax{b^{\rm max}}

\def\Imax{I^{\rm max}}

\def\supp{{\rm supp}}

\def\onebf{{\bf 1}}

\def\rest#1{_{\,\vrule height 1.5ex width 0.05em depth 1pt\, #1}}

\begin{document}
\parskip=2pt

\title{On one generalization of stable allocations in a two-sided market}

\author{Alexander V.~Karzanov
\thanks{Central Institute of Economics and Mathematics of
the RAS, 47, Nakhimovskii Prospect, 117418 Moscow, Russia; email:
akarzanov7@gmail.com.}
}
\date{}

 \maketitle
\vspace{-0.7cm}
 \begin{abstract}
In the stable allocation problem on a two-sided market introduced and studied by Baiou and Balinski in the early 2000's, one is given a bipartite graph $G=(V,E)$ with capacities $b$ on the edges (``contracts'') and quotas $q$ on the vertices (``agents''). Each vertex $v\in V$ is endowed with a linear order on the set $E_v$ of edges incident to $v$, which generates preference relations among functions (``contract intensities'') on $E_v$, giving rise to a model of \emph{stable allocations} for $G$. This is a special case of Alkan--Gale's stability model for a bipartite graph with edge capacities in which, instead of linear orders, the preferences of each ``agent'' $v$ are given via a \emph{choice function} that acts on the box $\{z\in\Rset_+^{E_v}\colon z(e)\le b(e),\, e\in E_v\}$ or a closed subset in it and obeys the (well motivated) axioms of \emph{consistence}, \emph{substitutability} and \emph{cardinal monotonicity}. By central results in Alkan--Gale's theory, the set of stable assignments generated by such choice functions is nonempty and forms a distributive lattice.

In this paper, being in frameworks of Alkan--Gale's model and generalizing the stable allocation one, we consider the situation when the preferences of ``agents'' of one side (``workers'') are given via linear orders, whereas the ones of the other side (``firms'') via integer-valued choice functions subject to the three axioms as above, thus introducing the model of \emph{generalized allocations}, or \emph{g-allocations} for short. Then the set $\Sscr$ of stable g-allocations is nonempty and forms a distributive lattice $(\Sscr,\prec)$. Our main aims are to characterize and efficiently construct \emph{rotations}, functions on $E$ associated with immediately preceding relations in $(\Sscr,\prec)$, and to estimate the computational complexity of constructing a poset generated by rotations for which the lattice of closed functions is isomorphic to $(\Sscr,\prec)$, obtaining a ``compact'' representation of the latter.
 \medskip

 \end{abstract}

%----------------------- Sec. 1

\section{Introduction}  \label{sec:intr}

Stable allocations in a two-sided market were introduced by Baiou and Balinski~\cite{BB} as a generalization of the notion of stable marriages from the prominent work by  Gale and Shapley~\cite{GS} and their natural extensions (of ``one-to-many'' and ``many-to-many'' types) widely studied in the literature.

Recall that in a setting of \emph{stable allocation model} (briefly, \emph{SAM}), one is given a bipartite graph $G=(V,E)$ in which the edges $e\in E$ are endowed with nonnegative real upper bounds, or \emph{capacities} $b(e)\in\Rset_+$, and the vertices $v\in V$ with \emph{quotas} $q(v)\in\Rset_+$. The vertices of $G$ are interpreted in applications as ``agents of the market'', and the edges as possible ``contracts'' between them, of which intensities, or allocation values, are restricted by the capacities and quotas. Formally, an \emph{allocation} is a nonnegative function $x:E\to\Rset_+$ obeying the capacity constraints: $x(e)\le b(e)$ for all $e\in E$, and the quota constraints: $x(E_v):=\sum(x(e)\colon e\in E_v)\le q(v)$ for all $v\in V$, where $E_v$ is the set of edges incident to $v$.

Besides, to estimate and compare the quality of allocations, each vertex $v\in V$ has a (strong) \emph{linear order} $>_v$ establishing \emph{preferences} on the edges in $E_v$; namely, $e>_v e'$ means that ``agent'' $v$ prefers ``contract'' $e$ to ``contract'' $e'$. Based on these preference relations, an allocation $x:E\to\Rset_+$ is said to be \emph{stable} if, roughly, there is no edge $e=\{u,v\}$ such that $x(e)$ can be increased, possibly decreasing $x$ on one edge $e'\in E_u$ with $e'<_u e$ and/or one edge $e''\in E_v$ with $e''<_v e$. 

Baiou and Balinski~\cite{BB} proved that a stable allocation always exists and can be found efficiently, by a strongly polynomial algorithm. Also they showed that if the capacities and quotas are integer-valued, then there exists an integral stable allocation, and revealed some other impressive properties. Subsequently, additional important results for SAM have been obtained. In particular, Dean and Munshi~\cite{DM} used a technique of \emph{rotations} to bijectively represent the stable allocations via the  so-called closed functions on a certain weighted poset which can be constructed efficiently. As an application, one can solve, in strongly polynomial time, the problem of minimizing a linear function over the set $\Sscr$ of stable allocations (in other words, finding a stable allocation of minimum cost, given a cost function on the edges).  (Note that originally nice properties of rotations and their applications have been revealed and deeply studied for the stable marriage and stable roommate problems in~\cite{irv,IL,ILG}.)

This paper is devoted to a certain generalization of SAM with integral capacities and quotas. We will rely on results in the fundamental work by Alkan and Gale~\cite{AG} on a general stability problem in two-sided markets.

Recall that they consider a bipartite graph $G=(V,E)$ and a closed subset $\Bscr$ of the box $\{z\in \Rset_+^E\colon z(e)\le b(e),\, e\in E\}$, where  $b\in\Rset_+^E$ (the entire box and its integer sublattice are typical samples of $\Bscr$). Each vertex $v\in V$ is endowed with a continuous \emph{choice function} $C_v$ on the restriction $\Bscr_v:=\Bscr\rest{E_v}$ that maps $\Bscr_v$ into itself so that $C_v(z)\le z$ for all $z\in\Bscr_v$. This choice function (CF) distinguishes the vectors $z\in\Bscr_v$ satisfying $C_v(z)=z$, called \emph{acceptable} ones, and establishes on them the so-called \emph{revealed preference} relation $\prec_v$, defined by $z'\prec_v z\Longleftrightarrow C_v(z\vee z')=z$ (where $z\vee z'$ takes the values $\max\{z(e),z'(e)\}$, $e\in E_v$). An assignment $x\in \Bscr$ for which all restrictions $x_v:=x\rest{E_v}$ ($v\in V$) are acceptable is said to be \emph{stable} if, roughly, there is no edge $e=\{u,v\}$ (``contract between agents $u$ and $v$'') such that one can increase $x$ at $e$ so as to obtain $x'\in\Bscr$ which is better than $x$ for both $u$ and $v$, in the sense that $C_u(x'_u)\succ_u x_u$ and $C_v(x'_v)\succ_v x_v$. 

In a general setting of~\cite{AG}, the CFs $C_f$ ($v\in V$) are assumed to obey the axioms of \emph{consistence} and \emph{substitutability} (also called \emph{persistence}), which earlier have been known and well motivated for simpler models, going back to Kelso and Crawford~\cite{KC}, Roth~\cite{roth}, Blair~\cite{blair}, and some other works. By adding the third axiom, of \emph{size-monotonicity} (going back to Fleiner~\cite{flein} and Alkan~\cite{alk}) or, stronger, of \emph{quota-filling}, Alkan and Gale succeeded to develop a rather powerful and elegant theory embracing a wide scope of stability problems in weighted bipartite graphs. A central result in it, generalizing earlier results for particular cases, is that the set of stable assignments (or stable schedule matchings, in terminology of~\cite{AG}) is nonempty and constitutes a \emph{distributive lattice} generated  by the above-mentioned relations $\prec_v$ for vertices $v$ occurring in one side of $G$. (Precise definitions and a brief account on results from~\cite{AG} needed to us will be given in Sect.~\SEC{defin}.) As a simple illustration of their model, Alkan and Gale point out just the stable allocation problem. We will refer to a general case of~\cite{AG} as the \emph{AG-model}.

(It should be noted, however, that some tempting issues are left beyond discussion in~\cite{AG}, such as: (i) to characterize (and/or explicitly construct) augmenting vectors (like rotations or so) determining ``elementary'' transformations in the lattice $(\Sscr,\prec)$ of stable assignments; or (ii) to represent $(\Sscr,\prec)$ via the lattice of closed (ideal-wise) functions on a poset which can be explicitly constructed or characterized. When $\Sscr$ is finite, such a poset exists by Birkhoff's theorem~\cite{birk} (though to efficiently construct it may take some efforts). For a general domain $\Bscr$ and general CFs $C_v$ in the AG-model, both problems (i) and (ii) look highly sophisticated. In some particular cases, e.g. those occurring in \cite{IL,ILG,BB,BAM,DM,FZ}, both problems are solved efficiently, in strongly polynomial time; here the role of augmenting vectors is played by certain edge-simple cycles, called rotations, which simultaneously serve as the elements of a poset in the representation of $(\Sscr,\prec)$. A different sort of rotations is demonstrated in~\cite{karz1} devoted to a sharper version of SAM where the vertices are endowed with \emph{weak} (rather than strong) linear orders; here both problems are solvable in strongly polynomial time as well, but for integer-valued $b$ and $q$, rotations may be formed by rational vectors having exponentially large numerators and denominators in the size of the graph.)

The stability model of our interest in this paper is a special case of the AG-model. Here, as above, we deal with a bipartite graph $G=(V,E)$ with a (fixed) partition of $V$ into two parts (viz. independent sets of $G$) denoted as $W$ and $F$ and interpreted as the sets of \emph{workers} and \emph{firms}, respectively (like in~\cite{AG}). Also: (a) each edge $e=\{w,f\}$ has an \emph{integer} capacity $b(e)\in\Zset_+$; (b) each $w\in W$ has an integer quota $q(w)\in\Zset_+$; and (c) the set $\Bscr$ of admissible assignments forms the integer lattice $\{x\in \Zset_+^E\colon x\le b\}$.

Besides, each vertex $v\in V$ is endowed with a choice function $C_v$ on $\Bscr_v$. These CFs are arranged differently for $W$ and $F$. Namely, for $w\in W$, $C_w$ is defined as in SAM, i.e. via a prescribed linear order $>_w$ on $E_w$. And for $f\in F$, like in the AG-model, $C_f$ is an arbitrary integer-valued CF on $\Bscr_f$ satisfying the consistence, substitutability and cardinal monotonicity axioms. In the latter case, we assume that $C_f$ is given via an \emph{oracle} that, being asked of $z\in\Bscr_f$, outputs the vector $C_f(z)$.

(In what follows, concerning computational complexity aspects, we will follow standard terminology (see e.g.~\cite{GJ}) and say that an algorithm that we deal with in our case is \emph{pseudo, weakly, strongly} polynomial if its running \emph{time} (viz. the amount of standard logical and arithmetical operations with numbers of binary size $O(\log (|E|\,\bmax))$) is estimated as $O(P\,\bmax)$, $O(P\,\log \bmax)$, and $O(P)$, respectively, where $\bmax:=\max\{b(e)\colon e\in E\}$ and $P$ is a polynomial in $|V|,|E|$. The term \emph{efficient} is applied to weakly and strongly polynomial algorithm. Note that we often do not care of precisely estimating time bounds of the algorithms that we devise and restrict ourselves merely by establishing their pseudo, weakly or strongly polynomial-time complexity.
As to the oracles computing values of CFs, we assume their efficiency. Moreover, we will liberally assume that the time of one application of an oracle is measured as a constant $O(1)$, like in a wide scope of problems where one is interested in the number of ``oracle calls'' rather than the complexity of their implementations.)

We call an acceptable assignment $x\in\Bscr$ a \emph{generalized allocation}, or a \emph{g-allocation} for short, and use abbreviation \emph{SGAM} for our stable g-allocation model. Since it is a special case of the AG-model, the set $\Sscr$ of stable g-allocations is nonempty and forms a distributive lattice under the relation $\prec_F$ defined for $x,y\in\Sscr$ by: $x\prec_F y \Longleftrightarrow x_f\prec_f y_f \;\forall f\in F$. Our study of SGAM utilizes some ideas from~\cite{karz2} on the corresponding boolean variant (viz. the special case of SGAM with all-unit capacities $b\equiv 1$).

Our first aim is to characterize the irreducible augmentations in the lattice $(\Sscr,\prec_F)$, by addressing the following problem: given $x\in\Sscr$, find the set $\Sscr_x$ of stable g-allocations $x'$ \emph{immediately succeeding} $x$ in the lattice, i.e. $x\prec_F x'$, and there is no $y\in\Sscr$ between $x$ and $x'$. We solve it efficiently (with $O(|E|)$ oracle calls) by constructing a set $\Lscr(x)$ of \emph{edge-simple} cycles in $G$ where each cycle $L$ is associated with a $0,\pm 1$ function $\chi^L$ on $E$ and determines the g-allocation $x':=x+\chi^L$ in $\Sscr_x$. This gives a bijection between $\Lscr(x)$ and $\Sscr_x$. The cycles in $\Lscr(x)$ are just what we call the \emph{rotations} applicable to $x$.

For each rotation $L\in\Lscr(x)$, we denote by $\tau_L(x)$ the maximal integer weight such that $x+\lambda\chi^L$ is a stable g-allocation for $\lambda=1,2,\ldots,\tau_L(x)$. Rotations and their maximal weights play an important role in an explicit representation of the lattice $(\Sscr,\prec_F)$. Here we follow an approach originally appeared in Irving and Leather~\cite{IL} (for stable marriages) and subsequently developed in some other researches (e.g. in Bansal et al.~\cite{BAM} for a ``many-to-many'' model, and in Dean and Munshi~\cite{DM} for SAM).

More precisely, we construct a sequence $x^1,\ldots,x^N$ of g-allocations along with rotations $L^1,\ldots,L^{N-1}$ and weights $\tau^1,\ldots,\tau^{N-1}\in\Zset_{>0}$ such that: $x^1$ and $x^N$ are, respectively, the minimal ($\xmin$) and maximal ($\xmax$) element of $(\Sscr,\prec_F)$; and for $i=1\ldots, N-1$, $x^{i+1}$ is equal to $x^i+\tau^i\, \chi^{L^i}$, where $L^i\in\Lscr(x^i)$ and $\tau^i=\tau_{L^i}(x^i)$. We call such a sequence a \emph{full route} (from $\xmin$ to $\xmax$). One shows that the family of pairs $(L,\tau)$ (respecting multiplicities) does not depend on a full route. Moreover, on the (invariant) family $\Rscr$ of rotation occurrences in a full route we arrange a transitive and antisymmetric binary relation $\lessdot$, obtaining a weighted poset $(\Rscr,\tau,\lessdot)$. One shows that there is a canonical bijection $\omega$ of $\Sscr$ to the set of closed functions on this poset, just giving the desired representation for $(\Sscr,\prec_F)$. 

(A function $\lambda:\Rscr\to\Zset_+$ not exceeding $\tau$ is called \emph{closed} if $L\lessdot L'$ and $\lambda(L')>0$ imply $\lambda(L)=\tau(L)$. A closed $\lambda$ generates the stable g-allocation $x=\omega^{-1}(\lambda)$ by setting $x:=\xmin+\sum(\lambda(L)\chi^{L}\colon L\in\Rscr)$. Here we write $\tau(L)$ for the maximal weight associated with (the corresponding occurrence of) a rotation $L$ in $\Rscr$.)

Note that there is an essential difference between full routes $x^1,\ldots,x^N$ occurring in usual and generalized allocation models. In case of SAM, the behavior is rather simple at two points. First, for each $i$, the weight $\tau^i$ of $L^i$ is computed easily, to be the minimum among the residual capacities $b(e)-x^i(e)$ when $\chi^{L^i}(e)=1$, and the values $x^i(e)$ when $\chi^{L^i}(e)=-1$. Second, all rotations $L^1,\ldots,L^{N-1}$ are different simple cycles and $N<|E|$. This implies that the elements of the representing poset can be identified with the rotations and the poset has size $O(|E|)$. Moreover, the rotations and their poset can be constructed in strongly polynomial time (an algorithm in~\cite{DM} has running time $O(|E|^2\log |V|)$).

In case of SGAM, the behavior is more intricate. First, given a rotation $L$ applicable to a stable $x$, in order to compute the maximal weight $\tau_L(x)$ efficiently, we are forced to use a sort of ``divide and conquer'' technique, which takes $O(|V| \log \bmax)$ oracle calls, resulting in a weakly polynomial procedure. Second, in the process of constructing a full route from $\xmin$ to $\xmax$, one and the same rotation can appear several times (even $O(\bmax)$ times), and as a consequence, the cardinality of the representing poset $\Rscr$ can include a factor of $\bmax$. We can illustrate such a behavior by an example in which the basic graph $G$ has only six vertices, but the size of the poset is of order $\bmax$, which can be arbitrarily large.

To overcome this trouble, we somewhat reduce the variety of possible CFs $C_f$ ($f\in F$) in our model by imposing an additional requirement on $C_f$ that we call the \emph{gapless} condition (a precise definition is given in Sect.~\SEC{addit_prop}). This implies the following property on rotations: if $x,x',x''$ are stable g-allocations such that $x\prec_F x'\prec_F x''$ and if a rotation $L$ is applicable to $x$ and $x''$, then $L$ is applicable to $x'$ as well.
The class of CFs obeying this condition (in addition to the three axioms as above) looks much wider compared with the one generated by linear orders on $E_f$ as in SAM (in particular, it includes all cases when $b(e)\le 2$ for each $e\in E$).

We explain that subject to the gapless condition, all rotations occurring in a full route are different and the number of rotations is estimated as $O(|V| |E|^2)$ (Theorem~\ref{tm:condC}). This implies that the weighted poset $(\Rscr,\tau,\lessdot)$ (where the lattice of closed functions is isomorphic to $(\Sscr,\prec_F)$) has size $|\Rscr|=O(|V| |E|^2)$. We construct this poset in weakly polynomial time (since a factor of $O(\log \bmax)$ arises in the procedure of finding the maximal weight of a rotation).

(Note that the availability of such a poset (which has size polynomial in $|E|$) gives rise to solvability, in strongly polynomial time, of the minimum cost stable g-allocation problem, by reducing it to the usual min-cut problem via Picard's method (like it was originally shown for stable matchings of minimum weight in~\cite{ILG}).)

This paper is organized as follows.
Section~\SEC{defin} contains basic definitions and settings. In Sect.~\SEC{act-rot} we describe the construction of rotations for our model SGAM and explain that they are closely related to irreducible augmentations in the lattice $(\Sscr,\prec_F)$ of stable g-allocations  (in Propositions~\ref{pr:xxp} and~\ref{pr:xpy}). Additional important properties of rotations are demonstrated in Sect.~\SEC{addit_prop}. Here we introduce the gapless condition~(C) on the choice functions $C_f$, $f\in F$, and show polynomial upper bounds on the number of rotations and the length of a full route under this condition (in Theorem~\ref{tm:condC}). Section~\SEC{poset_rot} describes the construction of poset $(\Rscr,\tau,\lessdot)$ induced by rotations and proves that $(\Sscr,\prec_F)$ is isomorphic to the lattice of closed functions on this poset (in Theorem~\ref{tm:omega}). 

Section~\SEC{construct} is devoted to algorithmic aspects, subject to  condition~(C). Here we first show that if the minimal g-allocation $\xmin$ is available, then the poset $(\Rscr,\tau,\lessdot)$ can be constructed in weakly polynomial time. Then we discuss  the problem of finding $\xmin$. Applying a general method from~\cite{AG} to solve such a problem in our case, we could reach merely a pseudo polynomial time bound. To obtain a better result, we develop an alternative approach. It consists of two stages. The first one (in Sect.~\SSSEC{stageI}) constructs a certain, not necessarily minimal, g-allocation $x$, and the second one (in Sect.~\SSSEC{stageII}) transforms $x$, step by step, into the required $\xmin$, by using a machinery of so-called reversed rotations that we elaborate. Both stages are carried out in weakly polynomial time. As a result, the whole task of constructing the poset $(\Rscr,\tau,\lessdot)$ under the gapless condition is solved efficiently (cf. Theorem~\ref{tm:time_poset}). The section finishes with a brief outline of an efficient method of finding a stable g-allocation of minimum cost (in Remark~2). 

The Appendix demonstrates an example of choice function on a 3-element set that obeys the basic axioms but violates the gapless condition. Using this, we compose a series of instances of SGAM with a fixed small graph so that the sizes of representing posets grow pseudo polynomially (estimated as $\Omega(\bmax)$). Finally, in Remark~3, we briefly discuss a method of constructing the representing poset in a general case of SGAM and estimate the complexity of this method.

%----------------------- Sec. 2

\section{Definitions and settings} \label{sec:defin}

In what follows we often (though not always) use terminology and definitions from~\cite{AG}. 

We consider a bipartite graph $G=(V,E)$ in which the vertex set $V$ is
partitioned into two parts (independent sets, color classes) $W$ and $F$,
called the sets of \emph{workers} and \emph{firms}, respectively. The edges
$e\in E$ of $G$ are endowed with nonnegative integer \emph{capacities}
$b(e)\in\Zset_{+}$, and the vertices (``workers'') $w\in W$ with
\emph{quotas} $q(w)\in\Zset_{+}$ (as to the ``firms'' $f\in F$, we do not
impose quotas to them at this moment).

One may assume, without loss of generality, that the graph $G$ is connected and
has no multiple edges. Then $|V|-1\le |E|\le \binom{|V|}{2}$. The
edge connecting vertices $w\in W$ and $f\in F$ may be denoted as $wf$.

For a vertex $v\in V$, the set of its incident edges is denoted by $E_v$. We
write $\Bscr$ for the integer box $\{x\in \Zset_+^E\colon x\le b\}$, and $\Bscr_v$ for its restriction to the set $E_v$, $v\in V$. A function (``assignment'') $x\in\Zset_+^E$ is called \emph{admissible} if it belongs to $\Bscr$ and satisfies the quota constraint $|x_w|=x(E_w)\le q(w)$ for each vertex $w\in W$.

(Hereinafter, for a numerical function $a$ on a finite set $S$, we write $a(S)$
for $\sum(a(e) : e\in S)$, and $|a|$ for $\sum(|a(e)|\colon e\in S)$. So
$a(S)=|a|$ if $a$ is nonnegative.)

Each vertex (``agent'') can prefer one admissible assignments to others. The
\emph{preferences} are arranged in different ways for ``workers'' and
``firms''.
 \smallskip

$\bullet$ (\textbf{linear preferences}) For $w\in W$, the preference relations are given via a \emph{linear order} $>_w$ on $E_w$. Here if $e,e'\in E_w$ satisfy
$e>_w e'$, then we say that $w$ prefers the edge $e$ to $e'$. This
gives rise to preference relations between the assignments on $E_w$ respecting
the quota $q(w)$, namely:
   \begin{numitem1} \label{eq:z_succ_zp}
for different $z,z'\in \Bscr_w$ with $q(w)=|z|\ge |z'|$, we write $z\succ_w z'$
if $e>_w \tilde e$ and $z(\tilde e)>0$ imply $z(e)\ge z'(e)$.
 \end{numitem1}

$\bullet$ (\textbf{choice functions}) For $f\in F$,  the preferences within
$\Bscr_f$ (defined later, in~\refeq{zzp}) are generated by a \emph{choice
function} (CF) $C=C_f:\Bscr_f\to \Bscr_f$ which satisfies $C(z)\le z$ for any $z\subseteq \Bscr_f$. Also $C$ obeys additional axioms. They concern pairs
$z,z'\in\Bscr_f$:
  \begin{itemize}
\item[(A1)] if $z\ge z'\ge C(z)$, then $C(z')=C(z)$ (\emph{consistence});
\item[(A2)] if $z\ge z'$, then $C(z)\wedge z'\le C(z')$ (\emph{substitutability}, or
\emph{persistence}).
  \end{itemize}

(Hereinafter, for $a,b\in\Rset^S$, the functions $a\wedge b$ (meet) and $a \vee b$ (join) take the values $\min\{a(e),b(e)\}$ and $\max\{a(e),b(e)\}$, $e\in S$, respectively.)

In particular, (A1) implies that any $z\in\Bscr_f$ satisfies $C(C(z))=C(z)$. As
is shown in~\cite{AG}, (A1) and (A2) imply that $C$ is \emph{stationary}, which
means that
 \begin{numitem1} \label{eq:plott}
any $z,z'\in\Bscr_f$ satisfy $C(z\vee z')=C(C(z)\vee z')$.
  \end{numitem1}

\noindent(In case of ordinary matchings (viz. 0,1 assignments), it is usually
called the  property of \emph{path independence} due to Plott~\cite{plott}.)

One more axiom imposed on each $C=C_f$, $f\in F$, known under the name of
\emph{size monotonicity},  requires that
 \begin{itemize}
 \item[(A3)] if $z\ge z'$, then $|C(z)|\ge |C(z')|$.
 \end{itemize}

An important special case of~(A3) is the condition of \emph{quota filling}; it
is applied when there is a prescribed \emph{quota} $q(f)\in \Zset_{>0}$, and is
viewed as:
 \begin{itemize}
 \item[(A4)]  any $z\in\Bscr_f$ satisfies $|C(z)|=\min\{|z|,q(f)\}$.
 \end{itemize}

Note that the above axioms are valid for the vertices $w\in W$ as well if each CF $C_w$ is generated by the linear order $>_w$ as
follows (cf. Example~1 from~\cite[Sect.~2]{AG}):
  \begin{numitem1} \label{eq:C_w}
for $z\in\Bscr_w$, (a) if $|z|\le q(w)$, then $C_w(z):=z$; (b) if $|z|> q(w)$, then, renumbering the edges in $E_w$ as $e_1,\ldots,e_{|E_w|}$ so that $e_i>_w e_{i+1}$ for each $i$, take the maximal $j$ satisfying $r:=\sum(z(e_i)\colon i\le j)\le q(w)$ and define $C_w(z)$ to be $(z(e_1),\ldots,z(e_j),q(w)-r,0,\ldots,0)$.
 \end{numitem1}

$\bullet$ (\textbf{stability}) For a vertex $v\in V$ and a function $x$ on $E$,
the restriction of $x$ to  $E_v$ is denoted as $x_v$. A function $z\in \Bscr_v$
is called  \emph{acceptable} if $C_v(z)=z$; the collection of such functions is
denoted by $\Ascr_v$. This notion is extended to $\Bscr$; namely, we
say that $x\in\Bscr$ is (globally) acceptable if $x_v\in\Ascr_v$ for all $v\in V$. The collection of acceptable assignments in $\Bscr$ is denoted
by $\Ascr$.

For any $v\in V$, the CF $C_v$ establishes preferences on acceptable
functions on $E_v$ as follows  (cf.~\cite{AG}): $z\in\Ascr_v$ is said to be (revealed) \emph{preferred} to $z'\in\Ascr'-\{z\}$ if
   \begin{equation} \label{eq:zzp}
   C_v(z\vee z')=z;
   \end{equation}
which is denoted as $z'\prec_v z$ (this matches~\refeq{z_succ_zp} for the
vertices in $W$). The relation $\prec_v$ is transitive (for $z\prec_v z'$
and $z'\prec_v z''$ imply $C_v(z\vee z'')=C_v(z\vee z'\vee
z'')=C_v(z'\vee z'')=z''$, in view of~\refeq{plott} applied to
$(z,z')$ and to $(z',z'')$).

The preferences between acceptable functions on the sets $E_v$, $v\in
V$, are extended to those on the whole $E$. Namely, for the side $F$ of $G$ and distinct $x,y\in \Ascr$, we write $x\prec_F y$ if $x_f\preceq_f y_f$ holds for all $f\in F$. The preferences in $\Ascr$ relative to the
``workers'' are defined in a similar way and denoted via $\prec_{\,W}$.

For $G$, $b$, $q(w)$, $>_w$ ($w\in W$) and $C_f$ ($f\in F$) as above, we will
refer to acceptable functions $x\in\Ascr$ on $E$ as \emph{generalized
allocations}, or \emph{g-allocations} for short. (When all CFs
$C_v$, $v\in V$, are generated by linear orders, they turn
into usual allocations introduced by Baiou and Balinski~\cite{BB}.)
 \medskip

\textbf{Definition 1.} For a vertex $v\in V$ and an acceptable function
$z\in\Ascr_v$, an edge $e\in E_v$ is called \emph{interesting} (in the sense
that some increase at $e$ could generate an assignment better than $z$) if
there exists $z'\in\Bscr_v$ such that
  \begin{equation} \label{eq:inter_e}
  z'(e)>z(e),\quad \mbox{$z'(e')=z(e')$ for all $e'\ne e$,}\quad \mbox{and $C_v(z')(e)>z(e)$}.
  \end{equation}
Extending this to acceptable g-allocations on $E$, we say that an edge
$e=wf\in E$ is \emph{interesting} for a vertex (``agent'') $v\in\{w,f\}$
under a g-allocation $x\in\Ascr$ if so is for $v$ under $x_v$. If  $e=wf\in E$
is interesting under $x$ for both vertices $w$ and $f$, then the edge $e$ is
called \emph{blocking} for $x$. A g-allocation $x\in \Ascr$ is called
\emph{stable} if no edge in $E$ blocks $x$. The set of stable g-allocations 
is denoted by $\Sscr=\Sscr(G,>,b,q,C)$.
\medskip

$\bullet$ ~For $v\in V$, the set $\Ascr_v$ endowed with the preference relation
$\succ_v$ turns into a lattice. In this lattice, for $z,z'\in\Ascr_v$, the
least upper bound (join) $z\curlyvee z'$ is expressed as $C_v(z\vee z')$ (while
the greatest lower bound (meet) $z\curlywedge z'$ is expressed by using a
certain closure operator; we omit this here). Cf.~\cite{AG}.
 \medskip

$\bullet$ ~In Alkan--Gale's general model in~\cite{AG}, given a bipartite graph $G=(V,E)$, the domain $\Bscr$ is a closed subset of the real box $\{x\in \Rset_+^{E}\colon x\le b\}$ for some $b\in \Rset_+^E$, and each choice function $C_v$ ($v\in V$) operates on the restrictions $x_v$ of $x\in\Bscr$ and obeys axioms (A1)--(A3). So our model of stable g-allocations, denoted as SGAM, is a special case of that, and general results in~\cite{AG} imply that the set $\Sscr$ of stable g-allocations possesses the following properties that will be important to us in what follows:
 \begin{numitem1} \label{eq:AG}
 \begin{itemize}
\item[(a)] $\Sscr$ is nonempty, and  $(\Sscr,\prec_F)$ is a
\emph{distributive} lattice (cf.~\cite[Ths.~1,8]{AG});
\item[(b)]
(\emph{polarity}): the order $\prec_F$ is opposite to 
$\prec_W$, namely: for $x,y\in\Sscr$, if  $x_f\prec_f y_f$ for all $f\in F$, then $y_w\prec_w x_w$ for all $w\in W$, and vice versa (cf.~\cite[Th.~4]{AG});
\item[(c)]
(\emph{unisizeness}): for each vertex $v\in V$, the size $|x_v|$ is
the same for all stable g-allocations $x\in \Sscr$ (cf.~\cite[Th.~6]{AG}).
 \end{itemize}
   \end{numitem1}

We denote the minimal and maximal elements in the lattice $(\Sscr,\prec_F)$ by
$\xmin$ and $\xmax$, respectively (then the former is the best and the latter
is the worst for the part $W$, in view of polarity~(2.3)(b)).

Note also that in case of quota filling (which, in particular, takes place for
the part $W$), property~(c) in~\refeq{AG} can be strengthened as follows
(cf.~\cite[Corollary~3]{AG}):
  \begin{numitem1} \label{eq:deficit}
for a vertex $v\in V$, if the choice function $C_v$ is quota filling (i.e.,
obeys axiom (A4)), and if some (equivalently, any) stable g-allocation $x$
satisfies $|x_v|<q(v)$, then the restriction $x_v$ is the same for all $x\in
\Sscr$.
 \end{numitem1}

In what follows, when no confuse can arise, we may write $\prec$ for
$\prec_F$.   Also for $e\in E'\subseteq E$, we write $\onebf^e_{E'}$ for the
unit base vector of $e$ in $\Rset^{E'}$ (which takes value 1 on $e$, and 0
otherwise). We usually omit the term $E'$ if it is clear from the context.

We finish this section with one more useful property.
  \begin{lemma} \label{lm:non-interest}
Let $v\in V$, $z,z'\in \Ascr_v$ and $z\prec_v z'$. Let $a\in E_v$ be such that
$z(a)\le z'(a)$ and $a$ is not interesting for $v$ under $z$. Then $a$ is not
interesting under $z'$ as well.
  \end{lemma}
   \begin{proof}
~Suppose, for a contradiction, that $a$ is interesting for $v$ under $z'$. Let
$C:=C_v$ and $p:=z'(a)$. Also define $y:=z+\onebf^{a}$ and $y':=z\vee
z'+\onebf^{a}$.  We have $C(z\vee z')=z'$ (since $z\prec_v z'$) and
$C(z'+\onebf^{a})(a)=p+1$ (since $a$ is interesting for $v$ under $z'$). Then
 \begin{multline*}
 C(y')(a)=C(z\vee z'+\onebf^{a})(a)= C(z\vee z'\vee (p+1)\onebf^{a})(a)  \\
 =C(C(z\vee z')\vee(p+1)\onebf^{a})(a)  =C(z'+\onebf^{a})(a)=p+1
 \end{multline*}
(using the fact that $z(a)\le z'(a)=p$). On the other hand, applying axiom (A2) to
$y'\ge y$ and using the fact that $a$ is not interesting for $v$ under $z$, we
have
  $$
 C(y')\wedge y\le C(y)=z.
  $$
But $C(y')(a)=p+1\ge z(a)+1$ and $y(a)=z(a)+1$; a contradiction.
   \end{proof}

%----------------------- Sec.3

 \section{Active graph and rotations} \label{sec:act-rot}

Throughout this section, we fix a stable g-allocation $x\in\Sscr$ different
from $\xmax$. We are interested in the set $\Sscr_x$ of stable g-allocations
$x'$ that are \emph{close} to $x$ and satisfy $x'\succ_F x$. This means that
$x$ \emph{immediately precedes} $x'$ in the lattice $(\Sscr,\prec_F)$; in other
words, there is no $y\in\Sscr$ between $x$ and $x'$, i.e. such that $x'\succ_F
y\succ_F x$. In order to find $\Sscr_x$, we construct the so-called
\emph{active graph} and then extract in it special cycles called
\emph{rotations}. Our method of constructing the rotations generalizes an
analogous method in~\cite{karz2} elaborated for the boolean variant of SGAM.
 \medskip

\textbf{Definition 2.} A vertex $w\in W$ is called \emph{deficit} if
$|x_w|<q(w)$. Otherwise (when $|x_w|=q(w)$)  $w$ is called \emph{fully filled};
the set of such vertices is denoted by $W^=$. An edge $e\in E$ is called
\emph{saturated} (by $x$) if $x(e)=b(e)$, and \emph{unsaturated} if
$x(e)<b(e)$.
  \medskip

\noindent(In light of unisizeness~\refeq{AG}(c), the set $W^=$ does not depend on $x\in\Sscr$. Also, by~\refeq{deficit}, for a deficit vertex $w$, the
restriction $x_w$ does not depend on $x\in\Sscr$.)

We will use the following
 \begin{lemma} \label{lm:e_interest}
Let $v\in V$ and $z\in\Ascr_v$. An edge $e\in E_v$ is interesting for $v$ under
$z$ if and only if

{\rm(i)} there exists $\tilde z\in\Ascr_v$ such that $\tilde z\succ_v z$ and
$\tilde z(e)>z(e)$; equivalently:

{\rm(ii)} there exists $\tilde z\in\Ascr_v$ such that $\tilde z\succ_v z$,  and
either {\rm(a)}~$\tilde z=z+\onebf^e$, or {\rm(b)}~$\tilde
z=z+\onebf^e-\onebf^{e'}$ for some $e'\in E_v-\{e\}$; equivalently,

{\rm(iii)} $C_v(z+\onebf^{e})\ne z$.
.
 \end{lemma}
  \begin{proof}
~Given an interesting $e$ for $v$ under $z$, let $z'$ be as in~\refeq{inter_e}.
Let $\tilde z=C(z')$, where $C=C_v$. Then $\tilde z(e)>z(e)$. Since $C$ is
stationary and $z'\vee z=z'$, we have
   $$
   \tilde z=C(z')=C(z'\vee z)=C(C(z')\vee z)=C(\tilde z\vee z),
   $$
whence $\tilde z\succ_v z$. Conversely, if $\tilde z\succ_v z$ and $\tilde
z(e)>z(e)$, then $z'$ defined by $z'(e):=\tilde z(e)$ and $z'(e'):=z(e')$ for $e'\ne e$ satisfies~\refeq{inter_e} (which is seen from $C(\tilde z\vee z)=\tilde z$ and $C(\tilde z\vee z)(e)\wedge z'(e)\le C(z')(e)$, by~(A2)). This gives the
assertion with~(i).

%if $\tilde z\succ_v z$ and $\tilde
%z(e)>z(e)$, then $z':=\tilde z\vee z$ satisfies~\refeq{inter_e}. This gives the
%assertion with~(i).

To see the assertion with~(ii), consider $z'$ as in~\refeq{inter_e} and let $z'':=z+\onebf^e$. Then $z''\le z'$, and applying axiom (A2), we have $C(z')\wedge z''\le C(z'')$. This implies that $\tilde z:=C(z'')$ satisfies $\tilde z(e)=z''(e)=z(e)+1$, whence $\tilde z\succ_v z$ (since $z'$ can be replaced by $z''$ in~\refeq{inter_e}). Applying~axiom (A3) to $z''>z$, we have $|\tilde z|=|C(z'')|\ge |C(z)|=|z|$. Also
$\tilde z(e')\le z''(e')=z(e')$ for all $e'\ne e$. Now the integrality of
$\tilde z$ implies either~(a) or~(b). The assertion with~(iii) is easy.
  \end{proof}

%-----------------Subsec.3.1

\subsection{Active graph} \label{ssec:act_graph}

For a fully filled vertex $w\in W^=$, let $\ell_w(x)$ denote the last
(least preferred under $>_w$) edge of the support $\supp(x_w):=\{e\in E_w\colon
x(e)\ne 0\}$. Consider the set $E'$ of edges $e=wf\in E_w$ satisfying
  \begin{numitem1} \label{eq:2prop}
(a) $e\le_w \ell_w(x)$, and (b) $e$ is interesting for $f$ under $x_f$.
  \end{numitem1}

In particular, $e$ is unsaturated. When $E'$ is nonempty, the most preferred edge in it
is called $W$-\emph{admissible} for $x$ and denoted as $a=a_w(x)$.

Consider this edge $a=wf$. By Lemma~\ref{lm:e_interest} applied to $z:=x_f$ and $a$, the function $C_f(z+\onebf^{a})$
is expressed as either $z+\onebf^{a}$, or $z+\onebf^{a}-\onebf^{c}$
for some $c\in E_f-\{a\}$. In the latter case, we say that the edge
$c$ is $F$-\emph{admissible} for $x$ (and $z$) and \emph{associated} with $a$, and call the pair $(a,c)$ a \emph{tandem} for $x$ (and $z$) passing the vertex $f$. So
  \begin{numitem1} \label{eq:tandem}
for $f\in F$ and $z\in\Ascr_f$, any tandem $(a,c)$ for $z$ passing $f$  satisfies
$z(a)<b(a)$, $c\in\supp(z)$, and $C_f(z+\onebf^a)=z+\onebf^a-\onebf^c$.
  \end{numitem1}

(Note that if $C_f(x_f+\onebf^{a})=x_f+\onebf^{a}$, then the edge $a$ generates no tandem. Also some $F$-admissible edges may be associated with two or more $W$-admissible ones, i.e. tandems $(a,c)$ and $(a',c')$ with
$a\ne a'$ and $c=c'$ are possible.) 

The following observation is useful: for $f\in F$ and $z\in\Ascr_f$ as above, 
  \begin{numitem1} \label{eq:acd}
if $(a,c)$ is a tandem for $z$ passing $f$, then $c$ is not interesting for $f$ under $z$.
  \end{numitem1}
Indeed, if $c$ is interesting for $f$, then either (a) $C(z+\onebf^c)=z+\onebf^{c}$, or (b) $C(z+\onebf^c)=z+\onebf^{c}-\onebf^{d}$ for some $d\ne c$, where
$C=C_f$. In case~(b), consider the functions $z':=z+\onebf^{a}$,
$z'':=z+\onebf^{c}$ and $y:=z+\onebf^{a}+\onebf^{c}$. Then
$C(z')=z+\onebf^{a}-\onebf^{c}$ and $C(z'')=z+\onebf^{c}-\onebf^{d}$. Applying axiom~(A2) to the inequalities $y>z'$ and $y>z''$, we have
   $$
   C(y)\wedge z'\le C(z')\quad\mbox{and}\quad C(y)\wedge z''\le C(z'').
   $$
Then $C(y)(c)\le z(c)-1$ and $C(y)(d)\le z(d)-1$. Also $C(y)(a)\le z(a)+1$ and
$C(y)(e)\le z(e)$ for all $e\ne a,c,d$. It follows that $|C(y)|\le
|z|-1$. But then $|C(y)|<|z|=|C(z')|$, contradicting axiom~(A3) since $y>z'$.

And in case~(a), for $y$ as above, one can see that $|C(y)|\le |z|$ and $|C(z+\onebf^c)|=|z|+1$, again obtaining a contradiction with~(A3) since $y>z+\onebf^c$.
  \medskip

Property~\refeq{acd} implies that no edge can be simultaneously $W$- and
$F$-admissible.

Now we form the directed graph $D=D(x)=(V,E_D)$ whose edge set consists of the $W$-admissible edges directed from $W$ to $F$ and  the $F$-admissible edges directed from $F$ to $W$; we use notation of the form $(w,f)$ for the former,
and $(f,w)$ for the latter edges (where $w\in W$ and $f\in F$). (Hereinafter,  depending on the context, an admissible edge may be thought of as an undirected edge in $G$ or as its directed counterpart in $D$.) We refer to $D(x)$ as the \emph{auxiliary} graph
for $x$.

Apply to $D$ the following procedure.

  \begin{description}
\item[\emph{Cleaning procedure}:]
When scanning a vertex $w\in W$ in a current auxilliary graph $D$, if we
observe that $w$ has leaving $W$-admissible edge $a=(w,f)$ but no entering
$F$-admissible edge (of the form $(f',w)$), then we delete $a$ from $D$.
Simultaneously, if such an $a$ has associated edge $c\in E_f$ and the latter edge does not occur in any other tandem passing $f$ in the current $D$, then we delete $c$ from $D$ as well. Along the way, we remove from $D$ each isolated (zero
degree) vertex whenever it appears. Repeat the procedure with the updated $D$,
and so on, until  $D$ stabilizes.
  \end{description}

Let $\Gamma=\Gamma(x)=(V_\Gamma,E_\Gamma)$ denote the graph $D$ upon
termination of the above procedure. We refer to $\Gamma$ as the \emph{active
graph} for $x$. Define $W_\Gamma:=W\cap V_\Gamma$ and
$F_\Gamma:=F\cap V_\Gamma$. For a vertex $v\in V_\Gamma$, denote the set of
edges of $\Gamma$ leaving $v$ (entering $v$) by
$\deltaout(v)=\deltaout_\Gamma(v)$ (resp. by $\deltain(v)=\deltain_\Gamma(v)$).
The graph $\Gamma$ possesses the following properties of ``balancedness'':
 \begin{numitem1} \label{eq:balanc}
(a) each vertex $w\in W_\Gamma$ satisfies $|\deltaout(w)|=|\deltain(w)|=1$; (b)
each vertex $f\in F_\Gamma$ satisfies $|\deltaout(f)|=|\deltain(f)|$, and the
tandems $(a,c)$ passing $f$ are edge disjoint and form a partition of the
set $\deltaout(f)\cup\deltain(f)$ (where $a\in \deltain(f)$ and $c\in \deltaout(f)$).
  \end{numitem1}

Indeed, from the definitions of $W$- and $F$-admissible edges we observe that
$|\deltaout(w)|\le 1$ for all $w\in W_\Gamma$, and
$|\deltain(f)|\ge|\deltaout(f)|$ for all $f\in F\cap V_D$. As a result of the
Cleaning procedure, for all $w\in W$, we obtain $|\deltain(w)|\ge|\deltaout(w)|$ (and
$|\deltaout(w)|\le 1$ as before), and for all $f\in F$, the sign $\ge$ in the above inequalities  preserves. Now the desired properties follow from the fact that the total number of entering edges (over the vertices in $\Gamma$) is equal to that of leaving ones.

%-----------------Subsec.3.2

 \subsection{Rotations} \label{ssec:rotat}

From~\refeq{balanc} it follows that the active graph $\Gamma=\Gamma(x)$ is
decomposed into a set of pairwise edge disjoint directed cycles, where each
cycle $L=(v_0,e_1,v_1,\ldots,e_k,v_k=v_0)$ is uniquely constructed in a natural
way, namely: for $i=1,\ldots,k$, if $v_i\in W_\Gamma$, then
$\{e_i\}=\deltain(v_i)$ and $\{e_{i+1}\}=\deltaout(v_i)$, and if $v_i\in
F_\Gamma$, then the pair $(e_i,e_{i+1})$ forms a tandem passing $v_i$. Note
that although all edges in $L$ are different, $L$ can be self-intersecting in
vertices of the part $F$ (so $L$ is \emph{edge-simple} but not necessarily
simple). Depending on the context,  the cycle $L$ may also be regarded as a
graph and denoted as $L=(V_L,E_L)$. Also when no confuse can arise, we identify directed edges in $\Gamma$ with their underlying undirected edges in $G$.

Let $\Lscr=\Lscr(x)$ be the set of above-mentioned cycles in $\Gamma$; these
cycles  are just what we call the \emph{rotations} generated by $x$. For each
$L\in\Lscr$, define $(L^+,L^-)$ to be the partition of $E_L$ where $L^+$
consists of the edges going from $W$ to $F$, called \emph{positive} edges, and
$L^-$ of the edges going from $F$ to $W$, called \emph{negative} ones. The
terms ``positive'' and ``negative'' will be applied to the corresponding
edges in the whole active graph $\Gamma$, as well as to their underlying undirected edges. 

Hereinafter, for a subset $A$ of edges of $\Gamma$, we denote the corresponding
incidence vector in $\Rset^E$ by $\chi^A$, i.e., for $e\in E$, $\chi^A(e)$ is 1
if $e$ is (the underlying undirected copy of) an edge from $A$, and 0 otherwise.

The key properties of rotations are exhibited in the following two
propositions.

 \begin{prop} \label{pr:xxp}
~For each $L\in\Lscr(x)$, the g-allocation $x':=x+\chi^{L^+} -\chi^{L^-}$ is
stable and satisfies $x'\succ x$ (where $\succ\,=\,\succ_F$).
 \end{prop}

We say that such an $x'$ is obtained from $x$ \emph{by applying the rotation}
$L$, or by \emph{shifting along} $L$,  with weight~1 and denote the set of these over all $L\in\Lscr(x)$ by $\Sscr_x$. It will be just the set $\Sscr_x$ mentioned in the beginning of this section.

 \begin{prop} \label{pr:xpy}
Let $y\in \Sscr$ and $x\prec y$. Then there exists $x'\in\Sscr_x$ such that
$x'\preceq y$.
  \end{prop}

The proofs of these propositions extend the corresponding ones  in~\cite{karz2} developed for the boolean variant of SGAM. They rely on the
next lemma.

\begin{lemma} \label{lm:acac}
Let $f\in F$. Let $(a(1),c(1)), \ldots,(a(k),c(k))$ be different (pairwise edge
disjoint) tandems in $\Gamma$ passing $f$. Then for any $I\subseteq
\{1,\ldots,k\}=:[k]$,
  \begin{multline*}
  C_f(x_f+\onebf^{a(1)}+\cdots+\onebf^{a(k)}-\sum(\onebf^{c(i)}\colon i\in I\}) \\
   =x_f+\onebf^{a(1)}+\cdots+\onebf^{a(k)}-\onebf^{c(1)}-\cdots-\onebf^{c(k)}.
  \end{multline*}
  \end{lemma}

 \begin{proof}
~Denote $x_f+\onebf^{a(1)}+\cdots+\onebf^{a(k)}-\sum(\onebf^{c(i)}\colon i\in
I\})$ by $z^I$. We have to show that $C(z^I)=z^{[k]}$ for any
$I\subseteq[k]$, where $C=C_f$.

First we show this for $I=\emptyset$. To this aim, we compare the actions of
$C$ on $z^\emptyset=x_f+\onebf^{a(1)}+\cdots+\onebf^{a(k)}$ and on
$y^i:=x_f+\onebf^{a(i)}$ for an arbitrary $i\in[k]$. The definition of tandem
$(a(i),c(i))$ implies $C(y^i)=x_f+\onebf^{a(i)}-\onebf^{c(i)}$. Applying
axiom~(A2) to the pair $z^\emptyset\ge y^i$, we have
  $$
  C(z^\emptyset)\wedge y^i\le C(y^i)=x_f+\onebf^{a(i)}-\onebf^{c(i)}.
  $$
Since $y^i(c(i))=x(c(i))$, it follows that $C_f(z^\emptyset)(c(i))<x(c(i))$.

Thus, $C(z^\emptyset)\le z^\emptyset-\onebf^{c(1)}-\cdots-\onebf^{c(k)}=z^{[k]}$. Here the inequality must turn into equality, since the size monotonicity axiom~(A3) applied to the pair $z^\emptyset\ge x_f$ gives
  $$
  |C(z^\emptyset)|\ge|C(x_f)|=|x_f|=|z^{[k]}|.
  $$

Now consider an arbitrary $I\subseteq[k]$. Then $z^\emptyset\ge z^I\ge
C(z^\emptyset)=z^{[k]}$. Applying axiom~(A1), we obtain $C(z^I)=z^{[k]}$,
as required.
 \end{proof}

This lemma implies that
 \begin{numitem1} \label{eq:zxab}
the function
$z:=x_f+\onebf^{a(1)}+\cdots+\onebf^{a(k)}-\onebf^{c(1)}-\cdots-\onebf^{c(k)}$
is acceptable for $f$ and satisfies $z\succ_f x_f$
  \end{numitem1}
(since Lemma~\ref{lm:acac} with $I=\emptyset$ gives
$C_f(z\vee x_f)=C_f(x_f+\onebf^{a(1)}+\cdots+\onebf^{a(k)})=z$).
 \medskip

\textbf{Proof of Proposition~\ref{pr:xxp}.} First of all we note that $x'_v$ is acceptable for all $v\in V$. This is immediate from the
acceptability of $x_v$ if $v\notin V_L$. As to $w\in W\cap V_L$, we have
$|x'_w|=|x_w|=q(w)$, implying $C_w(x'_w)=x'_w$. And for $f\in F\cap V_L$,
the acceptability of $x'_f$ follows from~\refeq{zxab} when $L^+\cap
E_f=\{a(1),\ldots,a(k)\}$ and $L^-\cap E_f=\{c(1),\ldots,c(k)\}$. Thus,
$x'\in\Ascr$.

Arguing from contradiction, suppose that $x'$ is not stable and consider a
blocking edge $e=wf$ for $x'$, i.e. $e$ is interesting for $w$ under $x'_w$,
and for $f$ under $x'_f$. Note that $x(e)\le x'(e)$. (For $x(e)>x'(e)$
would imply $e\in L^-\cap E_f=:B_f$. Then Lemma~\ref{lm:acac} applied to
$A_f:=L^+\cap E_f$ and $B':=B_f-e$ should give
$C_f(x'_f+\onebf^{e})=C_f(x_f+ \chi^{A_f}-\chi^{B'})=x'_f$, contrary to the
fact that $e$ is interesting for $f$ under $x'_f$.) Also the case $x(e)<x'(e)$
is impossible (for then we would have $e\in L^+$; so $e=a_w(x)$, whence
$e$ is not interesting for $w$ under $x'$). Therefore, $x(e)=x'(e)$.

Now suppose that the given edge $e$ is not interesting for $f$ under $x$. Then
$C_f(x_f+\onebf^{e})=x_f$, whence we obtain (using the stationarity~\refeq{plott} and~\refeq{zxab}):
 $$
 C_f(x'_f\vee (x_f+\onebf^{e}))=C_f(x'_f\vee C_f(x_f+\onebf^{e}))=C_f(x'_f\vee x_f)=x'_f.
 $$
On the other hand, for $z':=C_f(x'_f+\onebf^{e})$ and $d:=x(e)=x'(e)$, we have
  \begin{multline*}
 C_f(x'_f\vee (x_f+\onebf^{e}))= C_f(x'_f\vee x_f\vee (d+1)\onebf^{e}) \\
 =C_f(C_f(x'_f\vee x_f)\vee(d+1)\onebf^{e})= C_f(x'_f+\onebf^{e})=z'.
  \end{multline*}
Therefore, $z'=x'_f$. But in view of~Lemma~\ref{lm:e_interest}(i), the fact
that $e$ is interesting under $x'_f$ must imply $z'\succ_f x'_f$; a
contradiction.

Thus, $e$ is interesting for $f$ under $x$.

Now to come to the final contradiction, compare $e$ with the $W$-admissible
edge $a_w(x)$ and the edge $\ell_w(x)$ (defined in Sect.~\SSEC{act_graph}). Since  $e$ is interesting for $w$ under $x'$, we have $e>_w a_w(x)$. But since $e$ is interesting for $f$ under $x$, it is not
interesting for $w$ under $x$ (by the stability of $x$); hence $e\le_w \ell_w(x)$.
Thus, $e$ satisfies property~\refeq{2prop} and is more preferred than the
$W$-admissible edge $a_w(x)$; a contradiction. 
So $x'$ admits no blocking edges, and therefore is stable. The fact that $x'\succ x$ follows from~\refeq{zxab}. \hfill$\qed$
 \medskip

\textbf{Proof of Proposition~\ref{pr:xpy}.} In order to find a rotation
$L\in\Lscr(x)$ determining a required g-allocation $x'$, we first note that,
by~\refeq{AG}(c), the sizes of the restrictions of $x$ and $y$ for each vertex
$v\in V$ are equal: $|x_v|=|y_v|$. Since $x\ne y$, there exists a vertex $w\in
W$ such that $x_w\ne y_w$. Then $x\prec_F y$ implies $x_w\succ_w y_w$ (by 
polarity~\refeq{AG}(b)), and for $\ell_w=\ell_w(x)$ and $m_w:=\ell_w(y)$, we
have (cf.~\refeq{z_succ_zp}):
  \begin{equation} \label{eq:ell-ellp}
  \ell_w\ge_w m_w\quad \mbox{and}\quad y(m_w)>x(m_w).
  %\mbox{$x(e)\ge y(e)$ for all $e\in E_w$ with $e>_w \ell_w$},
  \end{equation}

Therefore, the edge $m_w=wf$ is not saturated by $x$, i.e., $x(m_w)<b(m_w)$.
Moreover, $m_w$ is interesting for the vertex $f$ under $x$. Indeed,
$y_f\succeq_f x_f$ implies $C(y_f\vee x_f)=y_f$, where $C:=C_f$. Letting
$\eps:=y(m_w)-x(m_w)$ and using the stationarity, we have
  $$
  y_f=C(y_f\vee x_f)=C((y_f-\eps\onebf^{m_w})\vee (x_f+ \eps\onebf^{m_w}))
   =C((y_f-\eps\onebf^{m_w})\vee C(x_f+\eps\onebf^{m_w})).
  $$
This implies that the value of $C(x_f+\eps\onebf^{m_w})$ on $m_w$ must be at
least $y(m_w)$. Then $C(x_f+\eps\onebf^{m_w})(m_w)>x(m_w)$ (since $\eps>0$),
whence $m_w$ is interesting for $f$ under $x$.

Thus, the edge $m_w$ satisfies the property as in~\refeq{2prop} under $x$. Then
$w$ has $W$-admissible edge $a_w=wg$ and this edge satisfies $\ell_w\ge_w
a_w\ge_w m_w$. Consider the vertex $g$ and the functions $x_g$ and $y_g$. The
edge $a_w$ is interesting for $g$ under $x$.
\medskip

\noindent\textbf{Claim.}
\emph{$C_g(x_g+\onebf^{a_w})=x_g+\onebf^{a_w}-\onebf^{c}$ for some $c\in
\supp(x_g)$. In addition: {\rm(i)} $y(c)\ne x(c)$, and {\rm(ii)}
$z:=x_g+\onebf^{a_w}-\onebf^{c}$ satisfies $x_g\prec_g z\preceq_g y_g$.}
 \medskip

 \begin{proof}
~By Lemma~\ref{lm:e_interest}(ii), since $a_w$ is interesting for $g$ under
$x_g$, we have $\tilde z:=C(x_g+\onebf^{a_w})\succ_w x_g$ and two cases are
possible: (a) $\tilde z=x_g+\onebf^{a_w}$, or~(b) $\tilde
z=x_g+\onebf^{a_w}-\onebf^{c}$ for some $c\in \supp(x_g)$, where $C:=C_g$.

%In case~(a), an analysis is easy when $C$ is quota filling. Indeed, $|\tilde z|>|x_g|$ and $C(\tilde z)=\tilde z$ imply that the vertex $g$ is deficit (see Definition~2). Then $y_g=x_g$ (by~\refeq{deficit}), and therefore, $a_w$ is interesting for $g$ under $y$. But $y(a_w)=x(a_w)$ implies $a_w>_w m_w$ (cf.~\refeq{ell-ellp}). Then $a_w$ is interesting for $w$ under $y$ as well, whence $a_w$ is blocking for $y$, contrary to the stability of $y$.

Let us show that case~(a) is impossible. First assume that $y(a_w)\ge x(a_w)+1$. Then $x_g+\onebf^{a_w}\le y_g\vee x_g$. At the same time, $|x_g|=|y_g|$ and
$C(y_g\vee x_g)=y_g$ imply
  $$
  |C(x_g+\onebf^{a_w})|=|x_g+\onebf^{a_w}|= |x_g|+1>|y_g|=|C(y_g\vee x_g)|,
  $$
contradicting axiom~(A3). Now assume that $y(a_w)\le x(a_w)=:d$. Then
 $$
  C(y_g\vee x_g+\onebf^{a_w})=C(y_g\vee x_g\vee (d+1)\onebf^{a_w})=
  C(C(y_g\vee x_g)\vee(d+1)\onebf^{a_w})= C(y_g\vee (d+1)\onebf^{a_w})
 $$
(taking into account that $d<b(a_w)$). Applying axiom~(A3) to
the inequality $y_g\vee x_g+\onebf^{a_w}\ge x_g+\onebf^{a_w}$, we obtain
  $$
  |C(y_g\vee (d+1)\onebf^{a_w})|=|C(y_g\vee x_g+\onebf^{a_w})|\ge |C(x_g+\onebf^{a_w})|=|x_g|+1=|y_g|+1.
  $$
This implies $C(y_g+\onebf^{a_w})\ne y_g$, whence the edge $a_w$ is interesting
for $g$ under $y$. But the inequality $y(a_w)\le x(a_w)$ is possible only if
$a_w\ne m_w$ (cf.~\refeq{ell-ellp}). Then $a_w>_w m_w$ and $y(a_w)<b(a_w)$
imply that $a_w$ is interesting for $w$ under $y$. So $a_w$ is blocking for
$y$; a contradiction.
 \smallskip

Thus, case~(b) takes place (and, obviously, $c\ne a_w$). Letting
$d:=x(a_w)$, we have
   $$
     z':=C(y_g\vee (d+1)\,\onebf^{a_w})=C(y_g\vee x_g\vee (d+1)\,\onebf^{a_w})
     =C((y_g-\onebf^{c})\vee (x_g+\onebf^{a_w}-\onebf^{c}).
    $$
It follows that $z'(c)<y(c)$. Then $y(a_w)\le x(a_w)$ (for otherwise,
$z'=C(y_g)= y_g$). This implies $a_w>_w m_w$, and we can see that
  $$
  C(y_g+\onebf^{a_w})=y_g+\onebf^{a_w}-\onebf^{c}.
  $$
But then $a_w$ is interesting under $y$ for both $w$ and $g$; a contradiction.

Thus, (i) in the Claim is valid. Next we show the relation concerning $z$
and $y_g$ in~(ii) (where $x_g\prec_g z$ is clear). Let $d:=x(a_w)$. First
suppose that $y(a_w)>d$. Then
  $$
  C(y_g)=C(y_g\vee x_g\vee (d+1)\onebf^{a_w})=C(y_g\vee C(x_g+\onebf^{a_w}))
    =C(y_g\vee z),
    $$
implying $z\preceq_g y_g$, as required. Now let $y(a_w)\le d$. Then $a_w>_w m_w$
(by~\refeq{ell-ellp}) and $y(a_w)<b(a_w)$ (since $x(a_w)<b(a_w)$). Therefore,
$a_w$ is interesting for $w$ under $y$. Also
  $$
  C(y_g\vee (d+1)\onebf^{a_w})=C(y_g\vee x_g\vee(d+1)\onebf^{a_w})=C(y_g\vee C(x_g+\onebf^{a_w}))=C(y_g\vee z).
  $$
If $C(y_g\vee z)=y_g$, we are done. And if  $C(y_g\vee z)\ne y_g$, then
$C(y_g\vee (d+1)\onebf^{a_w})\ne y_g$, whence $C(y_g+\onebf^{a_w})\ne y_g$ (cf.
Lemma~\ref{lm:e_interest}(iii)). But then $a_w$ is interesting under $y$ for
both $w$ and $g$, contrary to the stability of $y$. The Claim is proven. \end{proof}

Now we finish the proof of the proposition as follows. Let $c=w'g$. The Claim
implies that the edge $c$ forms a tandem with $a_w$. We have
$c\ge_{w'}\ell_{w'}(x)$ and $x_{w'}\ne y_{w'}$ (since $y(c)\ne x(c)>0$, by the
Claim). Then $w'$ has  $W$-admissible edge $a_{w'}(x)$; it satisfies
$\ell_{w'}(x)\ge_{w'} a_{w'}(x)\ge_{w'} m_{w'}$, where $m_{w'}:=\ell_{w'}(y)$. So
we can handle the edge $a_{w'}(x)$ in a way similar to that applied earlier to
$a_w$.

By continuing this process, we obtain an ``infinite'' path consisting of
alternating $W$-admissible and $F$-admissible edges in the auxiliary graph
$D(x)$. In this path, every pair of consecutive (underlying undirected) edges
incident to a vertex $f$ in $F$, say, $e,e'\in E_f$, forms a tandem such that:
$e\ne e'$, $x(e)<b(e)$, $e'\in\supp(x_f)$, and the function
$z:=x_f+\onebf^{e}-\onebf^{e'}\in\Ascr_f$ satisfies $z\preceq_f  y_f$, by the
Claim. Extracting from this path a minimal portion between two copies of the
same vertex in $W$, we obtain a cycle $L$ which is a rotation generated by $x$;
it determines the stable g-allocation $x':=x-\chi^{L^-}+\chi^{L^+}$, by
Proposition~\ref{pr:xxp} (where $L^+$ and $L^-$ are the sets of positive and
negative edges in $L$, respectively).

For $f\in F$, the cycle $L$ may pass $f$ several times; let
$(a(1),c(1)),\ldots,(a(k),c(k))$ be the tandems in $L$ passing $f$. By the
Claim, for $i=1,\ldots,k$, the function $z^i:=x_f+\onebf^{a(i)}-\onebf^{c(i)}$
satisfies $x_f\prec_f z^i\preceq_f y_f$. Then in the lattice
$(\Ascr,\succ_f)$, the join $\widehat z:=z^1\curlyvee_f \cdots \curlyvee_f z^k$
satisfies $x_f\prec_f \hat z \preceq_f y_f$. Now using Lemma~\ref{lm:acac},
we have
  $$
  z^1\curlyvee_f \cdots \curlyvee_f z^k=C_f(z^1\vee \cdots \vee z^k)=
  \vee_{i=1}^k(x_f+\onebf^{a(i)}-\onebf^{c(i)}).
  $$
Therefore, $\hat z$ coincides with the restriction $x'_f$ of $x'$ to $E_f$,
and we can conclude that the g-allocation $x'\in \Sscr_x$ satisfies $x\prec_F
x'\preceq_F y$, as required. \hfill$\qed\qed$

%----------------------- Sec.4

\section{Additional properties of rotations} \label{sec:addit_prop}

In this section we describe more properties of rotations in our model; they
will be useful to construct the poset of rotations in the next section.

Consider a rotation $L\in\Lscr(x)$ for a stable g-allocation $x$. In the
previous section we described the transformation of $x$ that consists in
increasing the values of $x$ by 1 on the set $L^+$, and decreasing by 1 on
$L^-$. The resulting g-allocation $x':=x+\chi^{L^+}-\chi^{L^-}$ is again stable
and more preferred for $F$: $x'\succ_F x$; in this case we say that $x'$ is obtained
from $x$ by \emph{shifting with weight 1 along} $L$. We, however,
can try to increase the weight of shifting along $L$.
\medskip

\noindent\textbf{Definition 3.} A weight $\lambda\in\Zset_{>0}$ is called
\emph{feasible} for a rotation $L\in\Lscr(x)$ if the following conditions hold:
 \begin{numitem1} \label{eq:feas_shift}
(a) $\lambda\le x(e)$ for each $e\in L^-$;\; (b) $\lambda\le b(e)-x(e)$ for
each $e\in L^+$;\; and
(c)~$C_f(x_f+\lambda\,\onebf^a)=x_f+\lambda\,\onebf^{a}-\lambda\,\onebf^{c}$
for each vertex $f\in F_L=F\cap V_L$ and each tandem $(a,c)$ in $L$ passing
$f$.
  \end{numitem1}

\begin{prop} \label{pr:feas_mu}
Let $\lambda$ be feasible for $L\in\Lscr(x)$, and let
$\mu\in\{1,2,\ldots,\lambda\}$. Then $x^\mu:=x+\mu\chi^{L^+}-\mu\chi^{L^-}$ is
a stable g-allocation.
 \end{prop}
  \begin{proof}
By~\refeq{feas_shift}(a),(b), we have $0\le x^\mu(e)\le b(e)$ for all $e\in E$,
i.e., $x^\mu\in\Bscr$. Also $|L^+\cap E_w|=|L^-\cap E_w|=1$ for each $w\in W$
implies that $x^\mu$ obeys the quota constraints on $W$. We start the proof with
showing that for a tandem $(a,c)$ in $L$ passing $f\in F$,
  \begin{equation} \label{eq:xfmu1}
C(x_f+\mu\,\onebf^{a})=x_f+\mu\,\onebf^{a}-\mu\,\onebf^{c},
  \end{equation}
where $C:=C_f$. To see this, for $\nu\in\{1,\mu,\lambda\}$, denote
$x_f+\nu\,\onebf^{a}$ by $y^\nu$. Then $y^\lambda\ge y^\mu\ge y^1$, implying
$|C(y^\lambda)|\ge |C(y^\mu)|\ge |C(y^1)|$, by~axiom (A3). This and
$|C(y^\lambda)|=|C(y^1)|=|x_f|$ imply $|C(y^\mu)|=|x_f|$. At the same time,
axiom (A2) applied to $y^\lambda\ge  y^\mu$ gives $C(y^\lambda)\wedge y^\mu\le
C(y^\mu)\le y^\mu$. This together with
$C(y^\lambda)=x_f+\lambda\,\onebf^{a}-\lambda\,\onebf^{c}$ and $\lambda\ge\mu$
implies $C(y^\mu)(a)=x(a)+\mu$ and $C(y^\mu)(e)=x(e)$ for all $e\in
E_f-\{a,c\}$. Now using $|C(y^\mu)|=|x_f|$, we obtain $C(y^\mu)(c)=x(c)-\mu$,
yielding~\refeq{xfmu1}.

For $f\in F$, let $(a(1),c(1)),\ldots,(a(k),c(k))$ be the tandems for $x$ occurring in $E_f\cap L$.
Put $\Delta:=\onebf^{a(1)}+\cdots+\onebf^{a(k)}$ and
$\nabla:=\onebf^{c(1)}+\cdots+\onebf^{c(k)}$. Using~\refeq{xfmu1} and arguing
as in the proof of Lemma~\ref{lm:acac}, one can see that
  \begin{equation} \label{eq:xmuDelta}
  C(x_f+\mu\Delta)=x_f+\mu\Delta-\mu\nabla \quad (=x_f^\mu).
  \end{equation}

Indeed, for $i=1,\ldots,k$, apply axiom (A2) to the inequality $z\ge y^{\mu,i}$,
where $z:=x_f+\mu\Delta$ and $y^{\mu,i}:=x_f+\mu\,\onebf^{a(i)}$. We obtain
$C(z)(c(i))\le C(y^{\mu,i})(c(i))=x_f(c(i))-\mu$ (using~\refeq{xfmu1}). This
implies $C(z)\le x_f^\mu$ (when $i$ varies). Since~axiom (A3) applied to $z>x_f$
gives $|C(z)|\ge|x_f|=|x_f^\mu|$,  equality~\refeq{xmuDelta} follows.

The next step in the proof is to show (using~\refeq{xfmu1}
and~\refeq{xmuDelta}) that for $f$ as above,
  \begin{numitem1} \label{eq:a(i)interest}
for $\mu=1,\ldots,\lambda-1$ and $i=1,\ldots,k$, the edge $a(i)$ is interesting
for $f$ under $x_f^\mu$, and there holds
$C(x_f^\mu+\onebf^{a(i)})=x_f^\mu+\onebf^{a(i)}-\onebf^{c(i)}$.
  \end{numitem1}
(extending the property that $a(i)$ is interesting for $f$ under $x_f$). To
show this, compare the functions $z=x_f+\mu\Delta$ (as before) with
$z':=x_f+\mu\Delta+\onebf^{a(i)}$ and $z'':=x_f+(\mu+1)\,\onebf^{a(i)}$.  Then
$z'>z$ and $z'\ge z''$. Applying~(A2) to $z'>z$ and using~\refeq{xmuDelta}, we
have
  $$
  C(z')\wedge z\le C(z)=x_f+\mu\Delta-\mu\nabla,
  $$
whence we obtain (a): $C(z')(e)\le (x_f+\mu\Delta-\mu\nabla)(e)$ for all $e\ne
a(i)$. In turn, axiom (A2) applied to $z'\ge z''$ gives (b): $C(z')(c(i))\le
C(z'')(c(i))=x(c(i))-\mu-1$ (using~\refeq{xfmu1} with $\mu+1$ and taking into
account that $x(c(i))\ge\lambda>\mu$, by~\refeq{feas_shift}(a)). Also, by~axiom (A3)
applied to $z'\ge z''$, we have (c): $|C(z')|\ge |C(z'')|=|x_f|$. From the
obtained relations (a),(b),(c) one can conclude that
  \begin{equation} \label{eq:Czp}
  C(z')=x_f+\mu\Delta-\mu\nabla+\onebf^{a(i)}-\onebf^{c(i)}=x_f^\mu+\onebf^{a(i)}-\onebf^{c(i)}.
  \end{equation}

Now let $y:=x_f^\mu+\onebf^{a(i)}$. Then $z'>y>C(z')$ (by~\refeq{Czp}), and
applying axiom~(A1), we obtain
$C(y)=C(z')=x_f^\mu+\onebf^{a(i)}-\onebf^{c(i)}$. This
gives~\refeq{a(i)interest}.

To finish the proof, consider the sequence $x=x^0,x^1,x^2,\ldots,x^\lambda$.
Initially, each $W$-admissible edge $a=wf\in L^+$ satisfies
$a=a_w(x)\le_w\ell_w(x)$, the g-allocation $x^1$ is obtained from $x$ by
shifting with weight 1 along the rotation $L$, and $x^1$ is stable by
Proposition~\ref{pr:xxp}. Under the update $x\mapsto x^1$, the edge $a$ as
above becomes the last edge in $\supp(x^1_w)$, i.e., $a=\ell_w(x^1)$. Moreover,
by~\refeq{a(i)interest} (with $\mu=1$), $a$ is interesting for $f$ under $x^1$,
and the tandem in $E_f$ involving $a$ preserves. Therefore, $a$ is the
$W$-admissible edge at $w$ under $x^1$, the previous rotation $L$ remains
applicable to $x^1$ as well, and the g-allocation $x^2$ obtained from $x^1$ by
shifting with weight 1 along $L$ is again stable (by Proposition~\ref{pr:xxp}
applied to $x^1$ and $L$). And so on. As a result, we obtain (by induction)
that each $x^\mu$ ($1\le\mu\le\lambda$) is stable. This completes the proof of
the proposition.
  \end{proof}

For a rotation $L\in\Lscr(x)$, we denote the \emph{maximal} feasible weight
$\lambda$ by $\tau_L=\tau_L(x)$. Also denote the set of tandems $(a,c)$ for
$x$ and $L$ passing a vertex $f\in F_L$ by $T_f(x,L)$. In light
of~\refeq{feas_shift}, $\tau_L$ can be computed efficiently, in the standard
``divide-and-conquer'' manner, namely, by use of the following iterative
procedure:
 \begin{itemize}
\item[(P):]
Initially assign lower and upper bounds for $\tau_L$ to be $\lambda:=1$ and
$\nu:=\min\{\min\{x(e)\colon e\in L^-\}, \min\{b(e)-x(e)\colon e\in L^+\}\}$,
respectively. Take the (sub)middle integer weight $\mu$ in the interval
$[\lambda,\nu]$, namely, $\mu:=\lfloor(\lambda+\nu)/2\rfloor$, and for each
vertex $f\in F_L$ and tandem $(a,c)\in T_f(x,L)$, compute the vector
$z_{f,a}:=C_f(x_f+\mu\,\onebf^{a})$. If each $z_{f,a}$ among these is equal to
$x_f+\mu\,\onebf^{a}-\mu\,\onebf^{c}$, then the lower bound is updated as
$\lambda:=\mu$. And if this is not so for at least one vector $z_{f,a}$, then
the upper bound is updated as $\nu:=\mu$. Then we handle in a similar way the
updated interval $[\lambda,\nu]$. And so on until we get $\lambda,\nu$ such
that $\nu-\lambda\le 1$.
 \end{itemize}

Then
 \begin{numitem1} \label{eq:log_iter}
procedure (P) terminates in at most $\log_2 b^{\rm max}$ iterations, where
$\bmax:=\max\{b(e)\colon e\in E\}$, and upon termination
of~(P), $\lambda$ or $\lambda+1$ is exactly $\tau_L(x)$.
 \end{numitem1}
Also using Proposition~\ref{pr:feas_mu} and Lemma~\ref{lm:acac}, one
can conclude that for $\lambda\in\Zset_{>0}$,
   \begin{numitem1} \label{eq:l<t}
if $\lambda<\tau_L(x)$, then shifting $x$ with weight $\lambda$ along $L$ preserves
the active graph, i.e., $\Gamma(x')=\Gamma(x)$ holds for
$x':=x+\lambda\chi^{L^+}-\lambda\chi^{L^-}$, implying $\Lscr(x')=\Lscr(x)$;
also $\tau_L(x')=\tau_L(x)-\lambda$ and $\tau_{L'}(x')=\tau_{L'}(x)$ for all
$L'\in\Lscr(x)-\{L\}$.
  \end{numitem1}

On the other hand, one can see that
  \begin{numitem1} \label{eq:l=t}
under the transformation of $x$ into the stable g-allocation
$x':=x+\tau_L(x)(\chi^{L^+}-\chi^{L^-})$, at least one of the following three
\emph{events} happens: (I) some $e\in L^-$ becomes 0-valued: $x'(e)=0$; (II)
some $e\in L^+$ becomes saturated: $x'(e)=b(e)$; or (III) there is $f\in F$
such that some tandem $(a,c)\in T_f(x,L)$ is destroyed, in the sense that
$\tau_L(x)<\min\{b(a)-x(a),\, x(c)\}$ but for
$z:=x_f+\tau_L(x)(\chi^{L^+}-\chi^{L^-})\rest{E_f}$, the vector
$C_f(z+\onebf^{a})$ is different from $z+\onebf^{a}-\onebf^{c}$.
  \end{numitem1}

Note that~\refeq{l<t} and~\refeq{l=t} imply that rotations with feasible
weights for $x$ commute. More precisely, using Lemma~\ref{lm:acac}, one can conclude with the following
\begin{corollary} \label{cor:commute}
Let $\Lscr'\subseteq\Lscr(x)$ and let $\lambda:\Lscr'\to\Zset_{>0}$ be such
that $\lambda(L)\le \tau_L(x)$ for each $L\in\Lscr'$. Then the g-allocation
$x':=x+\sum(\lambda(L)(\chi^{L^+}-\chi^{L^-})\colon L\in\Lscr')$ is stable,
each $L\in \Lscr'$ with $\lambda(L)<\tau_L(x)$ is a rotation for $x'$ having
the maximal feasible weight $\tau_L(x')=\tau_L(x)-\lambda(L)$, and each $L'\in
\Lscr(x)-\Lscr'$ is a rotation for $x'$ with $\tau_{L'}(x')=\tau_{L'}(x)$. In
particular, rotations in $\Lscr(X)$ can be applied in an arbitrary order.
\end{corollary}

Next we need one more notion.
 \medskip

\noindent\textbf{Definition 4.} Consider stable g-allocations $x\prec y$
(where, as before, $\prec=\prec_F$). Let $\Tscr$ be a sequence $x^1,\ldots,x^N$
of stable g-allocations such that $x^1=x$, $x^N=y$, and each $x^{i+1}$ ($1\le i<N$) is obtained from $x^i$ by shifting with a feasible weight $\lambda_i>0$ along a rotation
$L^i\in\Lscr(x^i)$. In particular, $x^1\prec\cdots\prec x^N$. We call $\Tscr$ a \emph{route} from $x$ to $y$, and may
liberally say that a rotation $L^i$ as above is \emph{used}, or
\emph{occurs}, in the route $\Tscr$. This route is called \emph{non-excessive}
if at each step of building $\Tscr$, the rotation used at this moment is taken with the maximal possible weight. Then $\lambda_i\le\tau_{L^i}(x^i)$ for all $i$, and the strict inequality is possible only if  $i=N-1$. A non-excessive route from $\xmin$ to $\xmax$ is called \emph{full}. (In this case the weight $\lambda_i$ of each rotation $L^i$ is maximal.)
 \medskip

Consider a non-excessive route $\Tscr=(x^1,\ldots, x^N)$ and let each $x^{i+1}$
be obtained from $x^i$ by shifting along a rotation $L^i$. We know that if a
vertex $w\in W$ is passed by a rotation $L^i$, then $x^i$ increases at the
$W$-admissible edge $a_w(x^i)$ and decreases at one edge $e\in E_w$ such
that $e>_w a_w(x^i)$. Furthermore, if $w$ is passed by a further rotation $L^j$
($j>i$), then the corresponding $W$-admissible edge at $w$ is either the same:
$a_w(x^j)=a_w(x^i)$, or less preferred: $a_w(x^j)<_w a_w(x^i)$. This implies
that, when moving along $\Tscr$, the values at each edge change in ``one-peak''
manner: they weakly increase, step by step, at a first phase, up to some
$x^i$,  and then weakly decrease at the second phase (each phase may be empty).
In particular, each edge becomes saturated at most once, and similarly it
becomes 0-valued at most once. Hence the total number of events I and II
(defined in~\refeq{l=t}) is at most $2|E|$.

Next we are interested in estimating the length $N$ of $\Tscr$. Based on the
above observations, this task reduces to estimating the number of consecutive
steps when an edge continues to be $W$-admissible (contained in current auxiliary graphs $D$). More precisely, for a fixed $e=wf\in F$, let
$L^{\alpha(1)},\ldots,L^{\alpha(k)}$, where $\alpha(1)<\cdots<\alpha(k)$, be
the sequence of rotations such that $e\in L^{\alpha(i)+}$; then $e$ is the
positive edge, $a$ say, in $E_w$ for $x^{\alpha(1)},\ldots, x^{\alpha(k)}$. Let
$(a,c^i)$ be the tandem containing this $a$ in $L^{\alpha(i)}$,
$i=1,\ldots,k$. Apriori for a general choice function $C_f$ (subject to
axioms (A1)--(A3)), it is possible that different elements among $c^1,\ldots,c^k$ can
be intermixed, in the sense that there are $i<j<p$ such that $c^i=c^p\ne c^j$.
For this reason, we cannot exclude the situation when one and the same rotation
(as a cycle) can appear in a non-excessive route more than once (possibly many
times), and as a result, the length of this route can be ``too large''. Such a behavior will be illustrated in the Appendix. (See also Remark~1 in the next section.)

We overcome this trouble by imposing an additional condition on each CF $C_f$,
$f\in F$. Here for $f\in F$, we refer to a pair $(a,c)$ in $E_f$ as an
(abstract) tandem for an admissible function $z\in\Ascr_f$ if $z(a)<b(a)$,
\;$z(c)>0$, and $C_f(z+\onebf^{a})=z+\onebf^{a}-\onebf^{c}$. The condition that
we wish to add, called the \emph{gapless condition}, reads as follows:
  \begin{itemize}
\item[(C):]
for each $f\in F$, if $z^1,z^2,z^3\in\Ascr_f$ and $a,c^1,c^2,c^3\in E_f$ are
such that: (i) $z^1\prec_f z^2\prec_f z^3$, (ii) $(a,c^i)$ is a tandem for
$z^i$, $i=1,2,3$, and (iii) $c^1=c^3$; then $c^1=c^2$.
  \end{itemize}

The simplest case of validity of (C) is when $C_f$ is generated by a linear
order on $E_f$ (i.e. we deal with the standard allocation model).
In this case, the relations $z^1\prec_f z^2\prec_f z^3$ imply $\ell_f(z^1)\le_f \ell_f(z^2)\le_f \ell_f(z^3)$, and for $i=1,2,3$, we have $a>_f\ell_f(z^i)$ and $c^i=\ell_f(z^i)$. Then $c^1=c^3$ implies $c^1=c^2=c^3$.

 \begin{theorem} \label{tm:condC}
Subject to the gapless condition, let $\Tscr=(x^1,\ldots,x^N)$ be a non-excessive
route. Then: {\rm(i)} any rotation occurs in $\Tscr$ at most once; and
{\rm(ii)} $|\Tscr|=O(|V| |E|^2)$.
  \end{theorem}
  \begin{proof}
~Let each $x^{i+1}$ be obtained from $x^i$ by shifting with weight $\lambda_i$
along a rotation $L^i$. Suppose that $L^i=L^k=:L$ for some $i<k$. The case
$k=i+1$ is impossible (for in this case one can combine the shifts $x^i\mapsto
x^{i+1}$ and $x^{i+1}\mapsto x^{i+2}$ into one shift $x^i\mapsto x^{i+2}$ of
weight $\lambda_i+\lambda_{i+1}$ along $L$, contrary to the non-excessiveness
of $\Tscr$). So $j:=i+1$ satisfies $i<j<k$.

Consider an edge $a=wf\in L^+$. Then $a=a_w(x^i)=a_w(x^k)$ and $x^i(a)<
x^j(a)\le x^k(a)$, implying $0<x^j(a)<b(a)$ and $a=\ell_w(x^j)$ (the last edge
in $\supp(x^j)$). Let $(a,c)$ be the tandem in $L$ passing $f$. Apriori three
cases are possible, letting $C:=C_f$: (1) $a$ is interesting for $f$ under
$x^j$ and $C(x^j+\onebf^{a})=x^j+\onebf^{a}-\onebf^{c}$; ~(2) $a$ is
interesting for $f$ under $x^j$ and
$C(x^j+\onebf^{a})=x^j+\onebf^{a}-\onebf^{c'}$ for $c'\ne c$; and~(3) $a$ is
not interesting for $f$ under $x^j$.

One may assume that $a\in L^+$ is chosen so that (1) is not the case. For if
the behavior as in case~(1) would take place for all edges in $L^+$, then $L$ should be a rotation for $x^j=x^{i+1}$ as well. But then one can apply to $x^i$ the rotation $L^i$ with a weight greater than $\lambda_i$, obtaining a contradiction with the non-excessiveness of $\Tscr$.

Case (2) is impossible as well, as it violates condition~(C) (with
$(z^1,z^2,z^3)=(x^i_f,x^j_f,x^k_f)$, \;$c^1=c^3=c$ and $c^2=c'$).

And in case~(3), we have $x^j(a)=x^k(a)$, and Lemma~\ref{lm:non-interest} (with
$z=x_f^j$ and $z'=x_f^k$) leads to a contradiction.

Thus, $L^i=L^k$ with $i\ne k$ is impossible, yielding assertion~(i) in the
proposition. Now to see~(ii), let $\beta(1)<\cdots< \beta(M)$ be 
such that under the transformations $x^{\beta(r)}\mapsto x^{\beta(r)+1}$ (and
only these), at least one $W$-admissible edge $a_w$ ($w\in W$) changes or
vanishes. Then $M<|E|$. And for each $r=1,\ldots,M$, the rotations
$L^{\beta(r)},\, L^{\beta(r)+1},\ldots, L^{\beta(i+1)-1}$ have the same set of
positive edges; denote it as $L^+$. Condition~(C) implies that in this sequence
of rotations, for each $a=wf\in L^+$, the tandems at $f$ involving $a$ can be
changed  at most $|E_f|-1$ times. Since after applying each of these rotations,
at least one tandem (over all $f\in F$) is changed and there are at most $|W|$
positive edges $a_w$ in these rotations, we obtain $\beta(r+1)-\beta(r)<|W|
|E|$. This implies $|\Tscr|<|V| |E|^2$, as required.
  \end{proof}

%----------------------- Sec.5

\section{Poset of rotations} \label{sec:poset_rot}

In this section we throughout assume validity of the gapless condition~(C) and, relying on
results in the previous section, construct the poset of rotations and show an
isomorphism  between the lattices of stable g-allocations and closed functions
for this poset. The idea of construction goes back to~\cite{IL,ILG} and was
extended subsequently to more general models (see e.g.~\cite{DM,FZ,karz2}).

If $\Tscr$ is a full route (see Definition~4), then, by
Theorem~\ref{tm:condC}(i), each rotation is used in $\Tscr$ at most once. A
nice property is that the set of rotations along with their weights does not
depend on the full route. Originally, a similar invariance property was
revealed by Irving and Leather~\cite{IL} for rotations in the classical stable
marriage problem and subsequently was shown by a number of authors for more
general models of stability. We show the property of this sort for our case in
the next lemma, considering an arbitrary pair of comparable stable
g-allocations, not necessarily $(\xmin,\xmax)$. Here $\Rscr(\Tscr)$ denotes the
set of rotations used in a non-excessive route $\Tscr$ (as before), and we
write $\Pi(\Tscr)$ for the set of pairs $(L,\lambda)$, where $L\in\Rscr(\Tscr)$ and $\lambda$ is the weight of $L$ applied in $\Tscr$. As before, $\prec$ stands for $\prec_F$.

  \begin{lemma} \label{lm:invar_rot}
Let $x,y\in\Sscr$ and $x\prec y$. Then for all non-excessive routes $\Tscr$
going from $x$ to $y$, the sets $\Pi(\Tscr)$ of pairs $(L,\lambda)$ used in
$\Tscr$ are the same.
 \end{lemma}
 \begin{proof}
Let $\Xscr$
denote the set of stable g-allocations $x'$ such that $x\preceq x'\preceq y$.
By Proposition~\ref{pr:xpy}, for each $x'\in\Xscr$, the set of routes from $x'$ to $y$ is nonempty. Let us say that $x'\in\Xscr$
is \emph{bad} if there exist two non-excessive routes $\Tscr,\Tscr'$ from $x'$
to $y$ such that $\Pi(\Tscr)\ne \Pi(\Tscr')$, and \emph{good} otherwise. One
has to show that $x$ is good (when $y$ is fixed).

Suppose this is not so, and consider a bad g-allocation $x'\in\Xscr$ of maximal height, in the sense that each g-allocation $z\in\Xscr$ with $x'\prec z\preceq
y$ is already good. In any non-excessive route from $x'$ to $y$, the first
g-allocation $z$  after $x'$ is obtained from $x'$ by applying a certain
rotation in $\Lscr(x')$ taken with the maximal possible weight. Then the
maximal choice of $x'$ implies that there are two rotations $L,L'\in\Lscr(x')$
along with their maximal possible (relative to~$y$) weights $\lambda,\lambda'$ (respectively)
such that the g-allocations $z$ and $z'$ obtained from $x'$ by shifting with
weight $\lambda$ along $L$, and with weight $\lambda'$ along $L'$
(respectively) belong to $\Xscr$ and are good, but there is a route $\Tscr$ from $x'$ to $y$ passing
$z$ and a route $\Tscr'$ from $x'$ to $y$ passing $z'$ for which
$\Pi(\Tscr)\ne\Pi(\Tscr')$.

At the same time, the pairs $(L,\lambda)$ and $(L',\lambda')$ commute at $x'$
(cf. Corollary~\ref{cor:commute}), i.e. $L$ with weight $\lambda$ is applicable
to $z'$, and $L'$ with weight $\lambda'$ is applicable to $z$ (and such weights are maximal possible). It follows that there are two
non-excessive routes $\tilde\Tscr$ and $\tilde\Tscr'$ going from $x'$ to
$y$ such that $\tilde \Tscr$ begins with $x',z,z''$, and $\tilde\Tscr'$ begins
with $x',z',z''$, after which these routes coincide; here $z''$ is obtained
from $z$ by applying $L'$ with weight $\lambda'$, or, equivalently, obtained
from $z'$ by applying $L$ with weight $\lambda$. Then
$\Pi(\tilde\Tscr)=\Pi(\tilde\Tscr')$. Since both $z,z'$ are good, there must be
$\Pi(\tilde\Tscr)=\Pi(\Tscr)$ and $\Pi(\tilde\Tscr')=\Pi(\Tscr')$. But then
$\Pi(\Tscr)=\Pi(\Tscr')$; a contradiction.
 \end{proof}

In particular, all full routes $\Tscr$ use the same set of rotations $\Rscr(\Tscr)$; we denote this set by $\Rscr$ (so $\Rscr$ consists of all
possible rotations applicable to stable g-allocations). Also each rotation $L$
is taken in any full route with the same (maximal) weight; we denote it by
$\tau_L$. A method of comparing rotations, originally demonstrated
in~\cite{IL}, works in our case as well.
 \medskip

\noindent\textbf{Definition 5.} ~For rotations $L,L'\in\Rscr$, we write
$L\lessdot L'$ and say that $L$ \emph{precedes} $L'$ if in \emph{all} full routes, the
rotation $L$ is used \emph{earlier} than $L'$.
 \medskip

In particular, if $L\lessdot L'$, then $L$ and $L'$ cannot occur in the same
set $\Lscr(x)$ (equivalently, the same active graph $\Gamma(x)$) of a stable
g-allocation $x\in\Sscr$.

The relation $\lessdot$ is transitive and anti-symmetric; so it determines a
partial order on $\Rscr$, forming the \emph{poset of rotations} for $G$. It is convenient to think of it as a weighted
poset in which each element $L\in\Rscr$ is endowed with the weight $\tau_L$; we
denote this poset as $(\Rscr,\tau,\lessdot)$. There is a relationship between
the lattice of closed functions on this poset and the lattice of stable
g-allocations (somewhat in spirit of Birkhoff~\cite{birk}), giving a ``compact
representation'' for the latter lattice.
 \medskip

\noindent\textbf{Definition 6.} ~A function $\xi:\Rscr\to\Zset_+$ satisfying
$\xi(L)\le\tau_L$ for all $L\in\Rscr$ is called \emph{closed} if $L\lessdot L'$
and $\xi(L')>0$ imply $\xi(L)=\tau_L$. A closed function $\xi$ is called
\emph{fully closed} if for each $L\in\Rscr$, the value $\xi(L)$ is equal to 0
or $\tau_L$.
 \medskip

(When $\tau$ is all-unit, a closed function is equivalent to an ideal of the
poset $(\Rscr,\lessdot)$.)

To establish a relationship between closed functions and stable g-allocations, for
each $x\in\Sscr$, take a non-excessive route $\Tscr$ from $\xmin$ to $x$,
consider the set $\Pi(\Tscr)$ of pairs $(L,\lambda)$ (cf.
Lemma~\ref{lm:invar_rot})  and define
  \begin{numitem1} \label{eq:omegax}
$\omega(x)$ to be the functions $\xi\in\Zset_+^\Rscr$ such that
$\xi(L)=\lambda$ if $(L,\lambda)\in \Pi(\Tscr)$, and $\xi(L)=0$ if $L$ does not
occur in $\Rscr(\Tscr)$.
  \end{numitem1}

By Lemma~\ref{lm:invar_rot}, $\xi$ depends only on $x$ (rather than $\Tscr$).
Moreover, $\xi$ is closed. (To see this, consider rotations $L\lessdot L'$ with
$\xi(L')>0$. Then $L'\in \Lscr(x')$ for some $x'\in\Tscr$, and $L$ occurs in
$\Tscr$ as well (earlier than $L'$). The situation when $L$ were used in
$\Tscr$ with a weight $\lambda$ smaller than $\tau_L$ would imply that $L$ is
contained in the active graph of $x'$, whence $L\in \Lscr(x')$ and $L,L'$ are
incomparable. Thus $\lambda=\tau_L$.)

The converse relation, saying that each closed function $\xi$ is the image by
$\omega$ of some $x\in\Sscr$, is valid as well; moreover, the map $\omega$
gives a lattice isomorphism.

 \begin{theorem} \label{tm:omega}
The map $x\mapsto \omega(x)$ establishes an isomorphism between the lattice
$(\Sscr,\prec_F)$ of stable g-allocations and the lattice $(\Xi,<)$ of closed
functions for $(\Rscr,\tau,\lessdot)$.
  \end{theorem}
  \begin{proof}
Consider stable g-allocations $x,y\in\Sscr$ and take their meet
$\alpha:=x\curlywedge y$ and join $\beta:=x\curlyvee y$ in the lattice
$(\Sscr,\prec_F)$. The core of the proof is to show that
  \begin{equation} \label{eq:meet-join}
\omega(x)\wedge \omega(y)=\omega(\alpha)\quad \mbox{and} \quad
\omega(x)\vee\omega(y)=\omega(\beta).
  \end{equation}
In other words, we assert that $\omega$ determines a homomorphism of the above
lattices.

To show the left equality in~\refeq{meet-join}, we consider the pairs
$\alpha\prec x$ and $\alpha\prec y$ and apply to each of them a method as in
the proof of Proposition~\ref{pr:xpy}. (One may assume that $\alpha\ne x,y$.)
Let $W_1$ ($W_2$) be the set of vertices $w\in W$ such that $\alpha_w\ne x_w$
(resp. $\alpha_w\ne y_w$). We assert that $W_1\cap W_2=\emptyset$.

To show this, choose an arbitrary initial vertex $w\in W_1$. Then
$\alpha_w\succ_w x_w$ (since $\alpha\prec_F x$).  By reasonings as in the proof
of Proposition~\ref{pr:xpy}, $\alpha$ has $W$-admissible edge
$a_w=a_w(\alpha)$ at $w$. Starting with $a_w$ and acting as in that proof, we
can construct the corresponding ``infinite'' path $P$ in the auxiliary graph
$D(\alpha)$ (so that the $W$- and $F$-admissible edges in $P$ alternate).
Extracting the first cycle appeared in $P$, we obtain a rotation
$L\in\Lscr(\alpha)$. Then the g-allocation $\alpha'$ obtained from $\alpha$ by
shifting with weight 1 along $L$ (i.e.
$\alpha':=\alpha+\chi^{L^+}-\chi^{L^-}$) is stable and satisfies
$\alpha\prec \alpha'\preceq x$.

Note that the path $P$ and rotation $L$ are constructed canonically; they are
determined by only the graph $D(\alpha)$ and the vertex $w$ (not
depending on $x$ in essence).

Suppose that $W_1\cap W_2\ne\emptyset$. Then taking a common vertex $w\in
W_1\cap W_2$, we would obtain the same rotation $L\in\Lscr(\alpha)$ for both
pairs $(\alpha,x)$ and $(\alpha,y)$. But then $\alpha'$ obtained (as above)
from $\alpha$ by shifting along $L$ satisfies $\alpha\prec\alpha'$ and
$\alpha'\preceq x,y$, whence $\alpha'\preceq x\curlywedge y=\alpha\prec
\alpha'$; a contradiction. Therefore, $W_1\cap W_2=\emptyset$.

Moving from $\alpha$ to $x$ in a similar way, we obtain a route $\Tscr_1$ from
$\alpha$ to $x$ such that for each $x'\in \Tscr_1$, the set of vertices
$w\in W$ with $ x'_w\ne x_w$ is a subset of $W_1$. And similarly for a route
$\Tscr_2$ from $\alpha$ to $y$ and the set $W_2$. It follows that
$\Rscr(\Tscr_1)\cap\Rscr(\Tscr_2)=\emptyset$. This implies the left equality
in~\refeq{meet-join}, taking into account that
$\omega(x)=\omega(\alpha)+\sum(\lambda\onebf^L \colon
(L,\lambda)\in\Pi(\Tscr_1))$ and $\omega(y)=\omega(\alpha)+\sum(\lambda\onebf^L
\colon (L,\lambda)\in\Pi(\Tscr_2))$, where $\onebf^L$ is the unit base vector
for $L$ in $\Rset^\Rscr$ (taking value 1 on $L$, and 0 otherwise). 

(For an arbitrary route $\Tscr$, we denote by $\Rscr(\Tscr)$ and $\Pi(\Tscr)$, respectively, the set of different rotations used in $\Tscr$, and the set of pairs $(L,\tau)$, where $L\in\Rscr(\Tscr)$ and $\tau$ is the total weight of (all occurrences of) $L$ in $\Tscr$.)

To see the right equality in~\refeq{meet-join}, consider a route $\Tscr_3$ from $x$ to $\beta$ and a route $\Tscr_4$ from $y$ to $\beta$, and form the concatenation $\Tscr$ of $\Tscr_1$ and $\Tscr_3$, and the concatenation $\Tscr'$ of $\Tscr_2$ and $\Tscr_4$ (both are routes from $\alpha$ to $\beta$). Then $\Pi(\Tscr)=\Pi(\Tscr')$ (by Lemma~\ref{lm:invar_rot}, taking into account that any route can be transformed into a non-excessive one preserving $\Rscr(\cdot)$ and $\Pi(\cdot)$). Also $r(\beta)-r(x)=r(y)-r(\alpha)$ and $r(\beta)-r(y)=r(x)-r(\alpha)$, where $r(x')$ denotes the rank (``distance'' from $\xmin$) of a stable g-allocation $x'$ in the distributive lattice $(\Sscr,\prec)$. These observations together with $\Rscr(\Tscr_1)\cap\Rscr(\Tscr_2)=\emptyset$ imply the equalities $\Pi(\Tscr_3)=\Pi(\Tscr_2)$ and $\Pi(\Tscr_4)=\Pi(\Tscr_1)$, whence the right equality in~\refeq{meet-join} easily follows.

Note also that from the proof of the left equality in~\refeq{meet-join}
one can conclude that the map $\omega$ is injective.

Finally, for a closed function $\xi$, consider the set (support)
$\Rscr_\xi:=\{L\in\Rscr\colon \xi(L)>0\}$ and let $L^1,\ldots,L^k$ be an order
on $\Rscr_\xi$ compatible with $\lessdot$. One can show by induction that the
sequence $\xmin=x^0,x^1,\ldots,x^k$, where
$x^i:=x^{i-1}+\xi(L^i)(\chi^{L^{i+}}-\chi^{L^{i-}})$, forms a correct
(non-excessive) route from $\xmin$ to $x=x^k$. Then $\omega(x)=\xi$, and
$\omega$ is surjective.

Thus, $\omega$ is bijective, and the result follows.
  \end{proof}
  
\noindent\textbf{Remark 1.} 
As is mentioned in Sect.~\SEC{addit_prop} (and will be shown in the Appendix), in a general case of SGAM (i.e. when the gapless condition~(C) is not imposed), the assertion as in~(i) of Theorem~\ref{tm:condC} may be wrong, namely, it is possible that a non-excessive route $\Tscr$ admits a rotation that occurs in $\Tscr$ several times. (Recall that $\Tscr$ is non-excessive if at each step of forming $\Tscr$ the current rotation (used at this moment) is applied with the maximal possible weight, except, possibly, for the last step.) In a general case, for a non-excessive route $\Tscr=(x^1,\ldots,x^N)$, let $\hat\Pi(\Tscr)$ denote the family of $N-1$ pairs $(L,\tau)$ determining the transformations $x^i\mapsto x^{i+1}$, i.e. such that $x^{i+1}=x^i+\tau(\chi^{L^+}-\chi^{L^-})$ (so $\hat\Pi(\Tscr)$ admits multiplications). Arguing in a spirit of the proof of Lemma~\ref{lm:invar_rot}, it is not difficult to obtain an analogous invariance property:
 \begin{numitem1} \label{eq:hatPi}
in a general case of SGAM, for any $x,y\in \Sscr$ with $x\prec y$, the family $\hat\Pi(\Tscr)$ is the same for all non-excessive routes from $x$ to $y$.
  \end{numitem1}
(Details of a proof are close to those for Lemma~\ref{lm:invar_rot} and we omit them here.) Note that $|\hat\Pi(\Tscr)|< \bmax|E|^2$, since the set of $W$-admissible edges changes (when moving through a route) at most $|E|$ times, and the number of consecutive steps when this set does not change is less than $\bmax|E|$. This will be used in Remark~3 in the Appendix.

%----------------------- Sec.6

\section{Efficient constructions} \label{sec:construct}

As before, we assume validity of condition~(C). Then the number of rotations $|\Rscr|$ is bounded by a
polynomial in the size of input graph $G$, by
Theorem~\ref{tm:condC}. By Lemma~\ref{lm:invar_rot}, in order to obtain all
rotations along with their maximal weights, it suffices to build, step by step,
an arbitrary full route. This procedure is rather straightforward when the initial g-allocation $\xmin$ (least preferred
for $F$) is known. However, the task of finding $\xmin$ itself is less trivial.
Another nontrivial task is to find the preceding relations $\lessdot$ on
rotations, which was defined implicitly (involving all full routes; see Definition~5 in Sect.~\SEC{poset_rot}).
In this section, we explain how to solve both problems efficiently. 
This gives an efficient method to construct the poset $(\Rscr,\tau,\lessdot)$.

%---------------------- SSEC 6.1

\subsection{Finding the preceding relation $\lessdot$ on rotations.}
\label{ssec:graph_H}

It suffices to construct the generating graph (viz. Hasse diagram)
$H=(\Rscr,\Escr)$ of the (unweighted) poset $(\Rscr,\lessdot)$. Here the set
$\Escr$ of directed edges is formed by the pairs $(L,L')$ in $\Rscr$ where $L$
\emph{immediately precedes} $L'$, i.e. $L\lessdot L'$ and no $L''\in\Rscr$
satisfies $L\lessdot L''\lessdot L'$. An efficient method to construct $H$ is
close to that developed for the boolean variant in~\cite[Sect.~5.1]{karz2}. It is
based on the following simple fact on the set of \emph{ideals} of
$(\Rscr,\lessdot)$ (i.e. subsets $I\subseteq\Rscr$ such that $L\lessdot L'$ and
$L'\in I$ imply $L\in I$).

  \begin{lemma} \label{lm:immed_preced}
For $L\in\Rscr$, let $\Imax_{-L}$ denote the maximal ideal of
$(\Rscr,\lessdot)$ not containing $L$. Let $I':=\Imax_{-L}\cup\{L\}$. Then $L$
immediately precedes a rotation $L'$ if and only if $L'$ is a minimal (w.r.t.
$\lessdot$) element in $\Rscr-I'$.
 \end{lemma}
 \begin{proof}
Clearly $\Imax_{-L}$ is the complement to $\Rscr$ of the set (``filter'')
$\Phi_L$ formed by the rotations greater than or equal to $L$. Then
the minimal elements of $\Phi_L-\{L\}$ are exactly the rotations $L'$ for which
$L$ is immediately preceding. These $L'$ just form the set of minimal elements
not contained in the ideal $I'$.
 \end{proof}

For $L\in\Rscr$, let $\Iscr^+_L$ denote the set of rotations immediately
succeeding  $L$. Based on Theorem~\ref{tm:omega} and
Lemma~\ref{lm:immed_preced}, we can construct the set $\Iscr^+_L$ of rotations
along with their maximal weights as follows (provided that the g-allocation
$\xmin$ is known).
  \medskip

\noindent\textbf{Constructing $\Iscr^+_L$.} This procedure consists of two
stages. At the \emph{first} stage, starting with $\xmin$, we build, step by
step, a non-excessive route $\Tscr$ using all possible rotations except for
$L$. Namely, at each step, for the current g-allocation $x$, we construct the
set $\Lscr(x)$ of rotations, choose an arbitrary rotation $L'\in\Lscr(x)$
different from $L$, compute its maximal possible weight $\tau_{L'}$ (as
described in procedure~(P) in Sect.~\SEC{addit_prop}) and shift $x$ with the
weight $\tau_{L'}$ along $L'$, obtaining a new current g-allocation. And so on.
In case $\Lscr(x)=\{L\}$, the first stage of the procedure finishes.

The \emph{second} stage consists in shifting the final g-allocation $x$ of
the first stage along $L$ taken with weight~$\tau_L$, obtaining a new
g-allocation $x'$. Constructing the set $\Lscr(x')$ of rotations generated by
$x'$, we just obtain the desired set $\Iscr^+_L$.
  \medskip

(Indeed, the g-allocation $x$ obtained upon termination of the first stage
corresponds to the ideal $I=\Imax_{-L}$, i.e. $\omega(x)=I$. The coincidence of
$\Iscr^+_L$ with $\Lscr(x')$ is obvious.)

Next we briefly discuss computational complexity aspects, relying on reasonings in
Sect.~\SEC{addit_prop}. A ``straightforward'' algorithm using the above
procedure to construct $H=(\Rscr,\Escr)$  consists of $|\Rscr|$ \emph{big
iterations}, each handling a new rotation $L$ in the list of all rotations.
(There is no need to form this list in advance, as it can be formed/updated in
parallel with carrying out big iterations. The list is initiated with the set
$\Lscr(\xmin)$.) Each big iteration solves the following task:
  \begin{itemize}
\item[(T):]
given $x\in\Sscr$, construct the active graph $\Gamma(x)$, decompose it into
rotations, obtaining the set $\Lscr(x)$, and compute the maximal feasible
weight $\tau_L$ for each $L\in \Lscr(x)$.
  \end{itemize}

To obtain $\Gamma(x)$, we first construct the auxiliary graph $D(x)$ by
scanning the vertices $w\in W$. For each $w$, we examine the edges $wf\in E_w$
by decreasing their preferences $>_w$, starting with the last edge $\ell_w(x)$
in $\supp(x_w)$. For each edge $e=wf$, we decide whether it is interesting for
$f$ under $x$ or not (by computing $C_f(x_f+\onebf^e)(e)$ and comparing it with
$x_f(e)$), and the first interesting edge (if exists) is output as the
$W$-admissible edge $a_w(x)$. Along the way, the $F$-admissible edges and tandems
are constructed. This gives the graph $D(x)$. The Cleaning procedure applied 
to $D(x)$ and a procedure of decomposing the obtained active graph $\Gamma(x)$ into rotations are routine and can be performed in time $O(|E|)$. Also for each rotations $L\in\Lscr(x)$, we use procedure~(P) to compute the maximal feasible weight $\tau_L$. In view of~\refeq{log_iter}, we obtain the following:
 \begin{numitem1} \label{eq:solvT}
task (T) is solvable in weakly polynomial time; it takes $O(|E|
|V|\log{\bmax})$ oracle calls and other (usual) operations.
 \end{numitem1}
(We assume that each choice function $C_f$ is given via  an ``oracle'' that,
being asked of a function $z\in \Bscr_f$, outputs $C_f(z)$, which takes one ``oracle call''.)

A routine method to construct the set $\Iscr^+_L$ for a fixed $L$
reduces to solving $|\Rscr|=O(|V| |E|^2)$ instances of  task (T). This gives a weakly
polynomial time bound, in view of~\refeq{solvT}. After
constructing the sets $\Iscr^+_L$ over all $L\in\Rscr$, we obtain the desired
set $\Escr$. This results in the following

  \begin{prop} \label{pr:time_for_H}
When $\xmin$ is available, the set $\Rscr$ of rotations $L$ along with their
maximal feasible weights $\tau_L$, and the minimal generating graph
$H=(\Rscr,\Escr)$ for the poset $(\Rscr,\lessdot)$ can be constructed in weakly
polynomial time, viewed as \,$\log{\bmax}$ times a polynomial in $|V|,|E|$.
 \end{prop}

Note that in the above ``naive'' method all sets $\Iscr^+_L$ are computed
independently, without storing intermediate data for further needs. In fact,
one can accelerate this method, at least by a factor of $|E|$, by acting in
spirit of an improved version for the boolean variant in~\cite[Sect.~5.1]{karz2}).
We omit details here and restrict ourselves with merely declaring the weakly
polynomial solvability.

%---------------------- SSEC 6.2

\subsection{Finding the initial stable g-allocation $\xmin$.} \label{ssec:xmin}

In order to construct $\xmin$, we can rely on a method in~\cite[Sec.~3.1]{AG}
intended for a general stability model on bipartite graphs considered there. Below we
outline that method regarding our model SGAM. This gives a
\emph{pseudo polynomial} algorithm for finding $\xmin$. Then we elaborate an
alternative approach to obtain an efficient, \emph{weakly polynomial}
algorithm. As a result, in view of Proposition~\ref{pr:time_for_H}, we obtain an
efficient algorithm to construct the poset $(\Rscr,\tau,\lessdot)$.

Acting as in~\cite{AG}, we iteratively construct triples $(b^i,x^i,y^i)$ of
functions on $E$, $i=0,1,\ldots,i,\ldots$. Initially put $b^0:=b$. In the
input of a current, $i$-th, iteration, there is a function $b^i\in \Zset_+^E$
(already known). The iteration consists of two stages.
 \smallskip

At the \emph{1st stage} of $i$-th iteration, $b^i$ is transformed into
$x^i\in\Zset_+^E$ by applying the CF $C_w$ to each restriction
$b^i_w={b^i}\rest{E_w}$ ($w\in W$), i.e. $x^i_w:=C_w(b^i_w)$ for all $w\in W$.
In other words, in our model with linear preferences $>_w$ and quotas $q(w)$
for $w\in W$, $x^i_w$ is obtained according to~\refeq{C_w} (with $b^i_w$ in
place of $z$). Then $x^i$ respects the quotas for all vertices of the part $W$.
 \smallskip

At the \emph{2nd stage} of $i$-th iteration, $x^i$ is transformed into $y^i\in
\Zset_+^E$ by applying the CF $C_f$ to each restriction $x^i_f=
{x^i}\rest{E_f}$ ($f\in F$), i.e. $y^i_f:=C_f(x^i_f)$ for all $f\in F$. Then
$y^i$ is acceptable for all vertices; however, it need not be stable.

The obtained functions $x^i,y^i$ are then used to generate the input function
$b^{i+1}$ for the next iteration. This is performed by the following rule:
  \begin{numitem1} \label{eq:Bi+1}
~for $e\in E$, (a) if $y^i(e)=x^i(e)$, then $b^{i+1}(e):=b^i(e)$; and (b) if
$y^i(e)<x^i(e)$, then $b^{i+1}(e):=y^i(e)$.
  \end{numitem1}

Then $b^{i+1}$ is transformed into $x^{i+1}$ and further into $y^{i+1}$ as described above. And so on. The process terminates when at the current, $p$-th say, iteration,
we obtain the equality $y^p=x^p$ (equivalently, $b^{p+1}=b^p$). Obviously,
$b^0> b^1>\cdots>b^p$, which implies that the process is finite and the number
of iterations does not exceed $|E|\,\bmax$. By a result in~\cite{AG}, the
following takes place.
 \begin{prop} \label{pr:Xp}
In the above algorithm, the final g-allocation $x^p$ is stable and optimal for $W$, i.e. $x^p=\xmin$.
 \end{prop}

(It should be noted that for a general model considered in~\cite{AG}, which can
deal with an infinite domain $\Bscr$, the process of iteratively transforming
the current functions can be infinite, though it always converges to some
triple $(\hat b,\hat x,\hat y)$ satisfying $\hat x=\hat y$. Moreover, one
proves (in Theorems~1 and~2 in~\cite{AG}) that the limit function $\hat x$ is
stable and optimal for the appropriate part of vertices. In our particulal
case, with integer-valued functions, the process is always finite, yielding
Proposition~\ref{pr:Xp}.)

However, the above algorithm provides merely a
pseudo polynomial time bound (where the number of iterations is estimated as
$O(|E|\,\bmax)$). Below we develop a ``faster'', weakly polynomial, algorithm
to find $\xmin$. It attracts ideas from the method of constructing a
non-excessive route from Sect.~\SEC{addit_prop}.

Our alternative algorithm consists of two stages.

%---------------------- SSSEC 6.2.1

\subsubsection{Stage I: Finding a stable g-allocation} \label{sssec:stageI}

The goal of this stage is to construct, in weakly polynomial time, a g-allocation which is stable but may differ from $\xmin$.

The stage consists of iterations. In the input of a current iteration, we are given $x\in\Zset_+^E$ such that:
  \begin{numitem1} \label{eq:mod_prop}
 $x\le b$; ~$|x_w|\le q(w)$ for all $w\in W$; and $x_f\in\Ascr_f$ (viz. $C_f(x_f)=x_f$)
for all $f\in F$.
 \end{numitem1}
 
Like in Sect.~\SSEC{act_graph}, let $W^=:=\{w\in W\colon |x_w|=q(w)\}$. 
For $w\in W$, let $\ell_w=\ell_w(x)$ be the last (w.r.t. $>_w$) edge in
$\supp(x_w)$ if $x_w\ne 0$, and the first edge in $E_w$ if $x_w=0$. Define
$A_w=A_w(x):=\{e\in E_w\colon e>_w\ell_w\}$. We assume that $x$ satisfies the following:
  \begin{numitem1} \label{eq:add_cond}
for $w\in W$, each edge $e=wf\in A_w$ is not interesting for $f$ under $x_f$, i.e. either $x(e)=b(e)$, or $x(e)<b(e)$ and $C_f(x_f+\onebf^e)=x_f$.
  \end{numitem1}

Initially we put $x:=0$. At an iteration,  the current (input) $x$ is handled as
follows. For $w\in W$, let $a=a_w(x)$ be the first edge $e=wf\in E_w$ such that: $e\le_w \ell_w$, and $e$ is interesting for $f$ under $x_f$ (cf.~\refeq{2prop}), letting $a_w(x)=\{\emptyset\}$ is such an $e$ does not exist. Using terminology as is Sect.~\SEC{act-rot}, we call this edge $a$ (when exists) $W$-\emph{admissible}, and if $C_f(x_f+\onebf^a)=x_f+\onebf^a-\onebf^c$ for some $c\in E_f-\{a\}$, then the edge $c$ is called $F$-\emph{admissible}, and $(a,c)$ a \emph{tandem} passing $f$.

Acting as in Sect.~\SSEC{act_graph}, we form the auxiliary directed graph $D=D(x)=(V_D,E_D)$, where each $W$-admissible ($F$-admissible) edge $wf$ induces the directed edge $(w,f)$ (resp. $(f,w)$) of $D$. Two cases are possible.
 \medskip
 
\noindent\underline{\emph{Case 1}}.
There is no (deficit) vertex $w\in W-W^=$ such that $a_w(x)\ne\{\emptyset\}$. Then $x$ is stable (and we are done). Indeed, suppose that $x$ admits a blocking edge $e=w'f'$. Then $e\notin A_{w'}$ (by~\refeq{add_cond}. So $e\le_{w'}\, \ell_{w'}$. But if $w'\in W^=$, then $e$ is not interesting for $w'$. And if $w'\in W-W^=$, then $a_{w'}(x)=\{\emptyset\}$ implies that $e$ is not interesting for $f'$.
  \medskip
  
\noindent\underline{\emph{Case 2}}. There is $w\in W-W^=$ such that $a_w(x)\ne\{\emptyset\}$. Choose such a $w$ and consider the maximal directed path $P=(v_0,e_1,v_1,\ldots, e_k,v_k,\ldots)$ in $D$ beginning at $v_0=w$ with the property that each pair $(e_i,e_{i+1})$ with $v_i\in F$ forms a tandem (this path is determined uniquely). Three subcases are possible: (a) $P$ is finite and its last vertex occurs in $F$; (b) $P$ is finite and its last vertex occurs in $W$; (c) $P$ contains a (unique) cycle $C$.  Let $L^+$ ($L^-$) be the set of (undirected copies of) $W$-admissible (resp. $F$-admissible) edges occurring in $P$ in subcases (a),(b), and in $C$ in subcase~(c).% (these $L^+,L^-$ are computed efficiently). 

Let $x':=x+\chi^{L^+}-\chi^{L^-}$. Arguing as in case of rotations in Sect.~\SSEC{rotat}, one shows that $x'_f\succ_f x_f$ for all $f\in F\cap V_L$, where $V_L$ is the set of vertices covered by $L:=L^+\cup L^-$. Also $x'$ maintains properties~\refeq{mod_prop} and~\refeq{add_cond} (a verification in detail is rather straightforward and left to the reader).

So, in Case~2, the current iteration includes: choosing a deficit vertex $w\in W-W^=$  with $a_w(x)\ne\{\emptyset\}$, constructing (completely or partially) the corresponding directed path $P$ beginning at $w$, extracting from $P$ the corresponding sets $L^+$ and $L^-$, and updating $x':=x+\chi^{L^+}-\chi^{L^-}$. The above-mentioned increase by $\succ_f$ for all $f\in F\cap V_L$ implies that the sequence of iterations is finite and the process terminates with Case~1, thus giving a stable g-allocation as required.

This algorithm is pseudo polynomial, and in order to accelerate it, we act as in Sect.~\SEC{addit_prop} and every time apply to the pair $(L^+,L^-)$ found at an iteration with input $x$ a ``divide-and-conquer'' procedure similar to that described in~(P). This computes the maximal weight $\tau$ such that $x':=x+\tau\chi^{L^+}-\tau\chi^{L^-}$ obeys~\refeq{mod_prop} and~\refeq{add_cond}, and we accordingly update $x\mapsto x'$. (In particular, in subcases~(a) and~(b) of Case~2, $\tau$ is bounded, among other restrictions, by the residual value $q(w)-|x_w|$.)

Thus, assuming validity of condition~(C), arguing as in the proof of Theorem~\ref{tm:condC}, and taking into account~\refeq{log_iter}, we can conclude with the following

\begin{prop} \label{pr:stage_I}
The above algorithm performing Stage~I finds a stable g-allocation in weakly polynomial time (which is polynomial in $|E|$ and linear in $\log \bmax$).
  \end{prop}
  %

%---------------------- SSSEC 6.2.2

\subsubsection{Stage II: Transforming $x\in\Sscr$ into $\xmin$} \label{sssec:stageII}

This is reduced to a series of the following tasks:
  \begin{numitem1}\label{eq:xp_prec}
given $x\in\Sscr$, find $x'\in\Sscr$ immediately preceding $x$ in $(\Sscr,\succ_F)$, i.e. $x$ is obtained by shifting $x'$ with weight 1 along a rotation $L'$ for $x'$, or establish that $x=\xmin$.
 \end{numitem1}
 
In turn, we say that $x'$ is obtained by shifting $x$ with weight 1 along the cycle reversed to $L'$, or along a \emph{reversed rotation}. Our method of constructing a reversed rotation is based on several observations below.

For $w\in W$, let $A_w=A_w(x)$ and $\ell_w=\ell_w(x)$ be defined as in Sect.~\SSSEC{stageI}. Then for $x'$ and $L'$ as above, the vertices $v$ such that $x'_v\ne x_v$ (equivalently, $v\in V_{L'}$) belong to $W^=$ (achieve the quota $q(v)$), and one can see that
  \begin{numitem1} \label{eq:xp-ne-x}
\begin{itemize}
 \item[(i)] for $w\in W$, if $x'_w\ne x_w$, then $x'_w-x_w=\onebf^c-\onebf^a$, where $a=\ell_w(x)$ and $c\in A_w(x)$ (in other words, $L'^+\cap E_w=\{a\}$ and $L'^-\cap E_w=\{c\}$);
 \item[(ii)] for $f\in F$, if $x'_f\ne x_f$, then $x'_f-x_f=\sum(\onebf^{c(i)}-\onebf^{a(i)}\colon i=1,\ldots,k)$, where $(a(1),c(1)),\ldots,(a(k),c(k))$ are the tandems for $x'$ passing $f$ that occur in the rotation $L'$.
 \end{itemize}
 \end{numitem1}
 
Define 
  \begin{gather*}
    U^-:=\{\ell_w\colon w\in W^=\}; \quad U^+_w:=\{e\in A_w\colon x(e)<b(e)\}\;\; (w\in W^=);\\
 U^+:=\cup(U^+_w\colon w\in W^=); \quad   U^-_f:=U^-\cap E_f\quad \mbox{and}\quad  U^+_f:=U^+\cap E_f\;\; (f\in F).
    \end{gather*}
    
In particular, $L'^-\subseteq U^+$ and $L'^+\subseteq U^-$ (in view of~\refeq{xp-ne-x} and $x'-x=\chi^{L'^-}-\chi^{L'^+}$).
 \medskip
 
\noindent\textbf{Definition 7.} For $w\in W^=$ and $c\in U^+_w$, we say that $(\ell_w(x),c)$ is a \emph{legal $w$-pair} for $x$. For $f\in F$, $a\in U^-_f$ and $c\in U^+_f$, we say that $(c,a)$ is a \emph{legal $f$-pair} for $x$ if $z:=x_f+\onebf^c-\onebf^a$ is acceptable, i.e. $C_f(z)=z$.
  \medskip
  
Note that each edge $c=wf\in U^+_w$ is interesting for $w$ under $x$ (since $x(c)<b(c)$ and, obviously, $C_w(x_w+\onebf^c)=x_w+\onebf^c-\onebf^{\ell_w}$). Therefore, the stability of $x$ implies that $c$ is not interesting for $f$ under $x$, i.e. $C_f(x_f+\onebf^c)=x_f$. Then for a legal $f$-pair $(c,a)$ and $z:=x_f+\onebf^c-\onebf^a$, we have
  $$
  C_f(z+\onebf^a)=C_f(x_f+\onebf^c)=x_f=z+\onebf^a-\onebf^c,
  $$
which means that $a$ is interesting for $f$ under $z$, $(a,c)$ is a tandem for $z$ passing $f$, and $x_f$ is obtained from $z$ by ``shifting along $(a,c)$''. In particular, $z\prec_f x_f$ (in view of $C_f(z\vee x_f)=C_f(x_f+\onebf^c)=x_f$).

Now for $f\in F$, fix a legal $f$-pair $(c,a)$ and consider an edge $d\in U^+_f$ different from $c$. This $d$ is not interesting for $f$ under $x$ (similar to $c$). The next fact is important.
 \begin{lemma} \label{lm:essent_pair}
Let $d\in U^+_f-\{c\}$ be interesting for $f$ under $z:=x_f+\onebf^c-\onebf^a$. Then $(d,a)$ is a legal $f$-pair and $(x_f+\onebf^d-\onebf^a)\succ_f z$.
  \end{lemma}
  \begin{proof}
Let $C=C_f$. Since $d$ is interesting under $z$, $C(z+\onebf^d)$ is equal to one of the following: (a) $z+\onebf^d$; (b) $z+\onebf^d-\onebf^{d'}$ for some $d'\ne c$; or (c) $z+\onebf^d-\onebf^c$.

Note that $C(x_f+\onebf^c)=x_f$ and $C(x_f+\onebf^d)=x_f$ imply that $\hat z:=x_f+\onebf^c+\onebf^d$ satisfies $C(\hat z)=x_f$. (For $\hat z>x_f+\onebf^c$ implies $C(\hat z)\wedge(x_f+\onebf^c)\le C(x_f+\onebf^c)=x_f$, whence $C(\hat z)(c)\le x_f(c)$. Similarly, $C(\hat z)(d)\le x_f(d)$, and now $C(\hat z)=x_f$ follows from $\hat z>x_f$ by~(A3).)

In case~(a), $\hat z>z+\onebf^d$ implies $|C(\hat z)|\ge|C(z+\onebf^d)|$ (by~(A3)). But $|C(\hat z)|=|x_f|$ and $|C(z+\onebf^d)|= |z+\onebf^d|=|x_f|+1$; a contradiction.

In case~(b), $\hat z>z+\onebf^d$ and $C(z+\onebf^d)=z+\onebf^d-\onebf^{d'}$ imply $C(\hat z)\wedge(z+\onebf^d)\le z+\onebf^d-\onebf^{d'}$ (by~(A2)). But $C(\hat z)(d')=x_f(d')$ and $(z+\onebf^d)(d')=x_f(d')$, whereas $(z+\onebf^d-\onebf^{d'})(d')<x_f(d')$ (since $d'\ne c,d$); a contradiction.

And in case (c), we have 
$$
C(z+\onebf^d)=z+\onebf^d-\onebf^c=(x_f+\onebf^c-\onebf^a)+\onebf^d-\onebf^c
=x_f+\onebf^d-\onebf^a,
$$
implying that $(d,a)$ is a legal $f$-pair. Also $C(z\vee(x_f+\onebf^d-\onebf^a))=C(z+\onebf^d)=x_f+\onebf^d-\onebf^a$, whence $(x_f+\onebf^d-\onebf^a)\succ_f z$, as required.  
  \end{proof}

For $f\in F$ and $a\in U^-_f$, let $U^+_f[a]$ denote the set of edges $c\in U^+_f$ such that $(c,a)$ is a legal $f$-pair. In light of the relations (by $\succ_f$) on such pairs established in Lemma~\ref{lm:essent_pair}, the following useful property takes place:
  \begin{numitem1}\label{eq:U+fa}
if $U^+_f[a]\ne\emptyset$, then there exists a legal $f$-pair $(c,a)$ such that no $d\in U^+_f-\{c\}$ is interesting for $f$ under $x_f+\onebf^c-\onebf^a$.
  \end{numitem1}
We say that a legal $f$-pair $(c,a)$ of this sort is \emph{essential}.

Now suppose that we succeed to construct a cycle $L=(v_0,e_1,v_1,\ldots,e_k,v_k=v_0)$ in $G$ such that: all edges $e_i$ are different; for $i$ even, $v_i\in W^=$ and $(e_{i},e_{i+1})$ is a legal $v_i$-pair (i.e. $e_{i}=\ell_{v_i}(x)$ and $e_{i+1}\in U^+_{v_i}(x)$); and for $i$ odd, $v_i\in F$ and $(e_{i},e_{i+1})$ is an \emph{essential} $v_i$-pair (where $e_{i}\in U^+_{v_i}[e_{i+1}]$ and $e_{i+1}\in  U^-_{v_i}$). We refer to $L$ as a \emph{legal cycle} for $x$ (which just plays the role of reversed rotation, in a sense), and denote by $L^+$ ($L^-$) the set of edges from $U^+$ (resp. $U^-$) in $L$. When needed, we also may think of $L$ as the corresponding subgraph $(V_L,E_L)$. Next we demonstrate the key property of our construction.
 \begin{lemma} \label{lm:legal_cycle}
Let $L$ be a legal cycle for $x$. Then $x':=x+\chi^{L^+}-\chi^{L^-}$ is stable and $x'\prec_F x$.
  \end{lemma}
  \begin{proof}
Suppose that $x'$ is not stable and let $e=wf$ be a blocking edge for $x'$. The fact that $e$ is interesting for $w$ under $x'$ implies that this is so for $w$ under $x$ as well. (The latter is immediate if $w\notin V_L$. And if $w\in V_L$, then, by the construction of $L$, the set $L^-\cap E_w$ consists of the only edge $\ell_w(x)$, and $L^+\cap E_w$ consists of one edge in $A_w(x)$. Then $e>_w\ell_w(x)$ and $x(e)\le x'(e)<b(e)$, whence $e\in U^+_w(x)$.) 

Thus, $e\in U^+_w(x)$ and $e$ is not interesting for $f$ under $x$ (since $x$ is stable). Let $(c(1),a(1)),\ldots,(a(k),c(k))$ be the legal $f$-pairs for $x$ occurring in $L$. Since each $(c(i),a(i))$ is essential, $e$ is not interesting for $f$ under $x_f+\onebf^{c(i)}-\onebf^{a(i)}$ (cf.~\refeq{U+fa}). 

For $i=1,\ldots,k$, define $z^i:=x_f+(\onebf^{c(1)}-\onebf^{a(1)})+\cdots+(\onebf^{c(i)}-\onebf^{a(i)})$. One shows that $e$ is not interesting for $f$ under each $z^i$, and that for any $j>i$, $(c(j),a(j)$ is an essential legal $f$-pair relative to $z^i$ (this is not difficult to show by induction on $i$, in spirit of the proof of Lemma~\ref{lm:acac}). 

As a consequence, $e$ is not interesting for $f$ under $x'_f=z^k$, contrary to the supposition. So $x'$ is stable. Now repeatedly applying Lemma~\ref{lm:essent_pair}, we obtain $x_f \succ z^1\succ \cdots\succ z^k=x'_f$, implying $x'\prec_F x$, as required.  
  \end{proof}
  
Lemma~\ref{lm:legal_cycle} enables us to conclude that $x$ is obtained from $x'$ by shifting (with weight 1) along the rotation $L'$ formed by ``reversing'' the cycle $L$, i.e. such that $L'^+=L^-$ and $L'^-=L^+$. So $x'$ immediately precedes $x$ in the lattice $(\Sscr,\succ_F)$.

Next we explain how to construct  a legal cycle $L$ for $x$ (to solve task~\refeq{xp_prec}). For this purpose, we form the following directed graph $\Gscr=\Gscr(x)$:
  \begin{numitem1} \label{eq:Gscr}
  \begin{itemize}
\item[((a)] 
each edge $a=wf\in U^-$ generates two vertices $w^a$ and $f^a$ and one (directed) edge $e^a=(f^a,w^a)$ in $\Gscr$;
 \item[(b)] each edge $c=wf\in U^+$ generates vertices $w^c$ and $f^c$ and edge $e^c=(w^c,f^c)$ in $\Gscr$;
 \item[(c)] each legal $w$-pair $(a,c)$ (where $w\in W$) generates edge $(w^a,w^c)$ in $\Gscr$;
 \item[(d)] each essential legal $f$-pair $(c,a)$ generates edge $(f^c,f^a)$ in $\Gscr$.
  \end{itemize}
  \end{numitem1}
  
One can see that each legal cycle for $x$ in $G$ determines, in a natural way, a simple directed cycle in $\Gscr$, and conversely. So the task of finding a legal cycle for $x$, or establishing that there is none (in which case we declare that $x=\xmin$) can be solved efficiently, in time $O(|V| |E|)$.

Finally, once a legal cycle $L$ is available, assuming validity of the gapless condition~(C) (Sect.~\SEC{addit_prop}), we can apply a ``divide-and-conquer'' procedure (cf.~(P) in Sect.~\SEC{addit_prop}) to find the maximal $\tau\in\Zset_{>0}$ such that $L$ is applicable to $x$ with weight $\tau$. (So $x+\tau\chi^{L^+}-\tau\chi^{L^-}$ is stable and the procedure takes $O(\log\,\bmax)$ steps.)

As a result, starting with an arbitrary stable g-allocation $x$ and acting as above, step by step, we are able to construct a sequence $x^0,x^1,\ldots,x^N$ in $\Sscr$ such that $x^0=x$, ~$x^N=\xmin$, and the reversed sequence $x^N,\ldots, x^0$ forms a non-excessive route from $\xmin$ to $x$. Then $N=O(|V| |E|^2)$ (cf. Theorem~\ref{tm:condC}) and we come to the following result:
  \begin{numitem1} \label{eq:stage_II}
the above algorithm for Stage~II transforms an arbitrary $x\in \Sscr$ into $\xmin$ efficiently, in time $\log{\bmax}$ times a polynomial in $|E|$.
  \end{numitem1}

This together with Propositions~\ref{pr:stage_I} and~\ref{pr:time_for_H} gives the following
  \begin{theorem} \label{tm:time_poset}
Subject to condition~(C), the poset of rotations $(\Rscr,\tau,\lessdot)$ can be constructed in weakly polynomial time, estimated as $\log{\bmax}$ times a polynomial in $|E|$ (including oracle calls).
 \end{theorem}
 
\noindent\textbf{Remark~2.}
Suppose we address the \emph{problem of finding a stable g-allocation of minimum cost} ($\ast$):
given \emph{costs} $c(e)\in\Rset$ of edges $e\in E$, find $x\in\Sscr$ minimizing the total cost $cx:=\sum(c(e)x(e)\colon e\in E)$. Subject to condition~(C), the size $|\Pscr|$ of poset $(\Pscr,\tau,\lessdot)$ representing $\Sscr$ is polynomial in $|E|$ (by Theorem~\ref{tm:condC}). This enables us to solve problem~($\ast$) efficiently, in strongly polynomial time, when the poset $(\Pscr,\tau,\lessdot)$ %and the minimal stable g-allocation $\xmin$ are 
is explicitly known. This is due to a nice approach originally elaborated in Irving et al.~\cite{ILG} to minimize a linear function over stable marriages and subsequently applied to some other models of stability; it consists in a reduction to a linear minimization on ideals or closed functions for a related poset, and further  to the standard minimum cut problem, by a method due to Picard~\cite{pic}. We briefly outline how it works in our case of SGAM (referring for more details, e.g., to~\cite[Sec.~6]{karz2}).

First we compute the cost $c_L:=c(L^+)-c(L^-)$ of each rotation $L\in\Rscr$. Then for each $x\in\Sscr$ and the corresponding closed function $\xi=\omega(X)$ (cf.~\refeq{omegax}), we have $cx=c\xmin+\sum(c_L\xi(L)\colon L\in \Rscr)$.
This gives the problem equivalent to~($\ast$), namely ($\ast\ast$): find a closed function $\xi$ minimizing the total weight $c^\xi:=\sum(c_L\xi(L)\colon L\in\Rscr)$. Clearly ($\ast\ast$) has an optimal solution determined by an ideal $\Iscr$ of $(\Rscr,\lessdot)$ (in which case $\xi(L)=\tau_L$ if $L\in \Iscr$, and 0 otherwise).

To solve ($\ast\ast$), we act as in~\cite{pic} and transform the generating graph $H=(\Rscr,\Escr)$ of the poset into directed graph $\hat H=(\hat\Rscr,\hat\Escr)$ with edge capacities $h$, as follows:

(a) add two vertices: ``source'' $s$ and ``sink'' $t$; 

(b) add the set $\Escr^+$ of edges $(s,L)$ for all vertices $L$ in $\Rscr^+:=\{L\in \Rscr\colon c_L>0\}$;

(c) add the set $\Escr^-$ of edges $(L,t)$ for all vertices $L$ in $\Rscr^-:=\{L\in\Rscr\colon c_L<0\}$;

(d) assign the capacities $h(s,L):=c_L\tau_L$ for $(s,L)\in \Escr^+$,  $h(L,t):=|c_L\tau_L|$ for $(L,t)\in \Escr^-$, and $h(L,L'):=\infty$ for $(L,L')\in\Escr$.

Now it is rather straightforward to show that an ideal $\Iscr$ of $(\Rscr,\lessdot)$ determines an optimal solution to ($\ast\ast$) if and only if the set of edges of $\hat H$ going from $\{s\}\cup(\Rscr-\Iscr)$ to $\{t\}\cup\Iscr$ forms a minimum capacity cut among those separating $s$ from $t$.

As a consequence, the polynomial size of $\Pscr$ implies that each of the above problems: the min-cut one, ($\ast\ast$) and ($\ast$), is solvable in strongly polynomial time, as required.

 \appendix

\section{Appendix. Withdrawal of property (C) and large rotational posets} \label{sec:notC}

According to Theorem~\ref{tm:condC}, if the choice functions $C_f$ for all
vertices $f\in F$ satisfy the gapless condition~(C), then any rotation (regarded as a
cycle) is involved in a non-excessive route at most once, and the number of
rotations is estimated as $O(|V| |E|^2)$. This implies that the size of the
rotational poset (without weights) is bounded by a polynomial in the size of
the input graph $G$. In this section we construct examples showing that
discarding condition~(C) (while preserving axioms (A1)--(A4)) may cause a
significant grow of the size of the corresponding poset, which may be proportional to
the maximal edge capacity (upper bound) $\bmax$. Moreover, in a non-excessive
route, one and the same cycle of $G$ may appear as a rotation many times.

We will use one construction of choice functions that was
pointed out to the author by Vladimir Danilov (a private communication).
%(Author's original construction of appropriate choice functions looks more intricate and we omit it here.)
In fact, the choice functions of this sort are
generated by certain extensions of linear orders.

This construction associates to arbitrary $k\in\Zset_{>0}$ and $b\in
\Zset_+^{[k]}$ a planar \emph{diagram} $D=D_{k,b}$ formed by unit square
\emph{cells} labeled by coordinates $(i,j)$, where $i\in[k]$ (=$\{1,\ldots,k\}$) and $j$ ranges $0,1,\ldots,b(i)$. We think of the sequence of cells
$(i,0),(i,1),\ldots,(i,b(i))$ as $i$-th \emph{column} in $D$ (growing upward),
denoted as $\zeta_i$, and refer to a subset of cells as a \emph{section} of
$D$ if it contains exactly one element in each column.

Let $\Bscr=\Bscr_{k,b}:=\{z\in \Zset_+^{[k]}\colon z\le b\}$. For $z\in\Bscr$,
define the \emph{section} $\sigma(z)$ to be the set $\{(i,j)\colon i\in[k],\,
j=z(i)\}$ (so $\sigma$ gives a bijection between $\Bscr$ and the set of
sections in $D$). Also for $z\in\Bscr$, we denote by $\Lambda_z$ (and by
$\Lambda_{\sigma(z)}$) the set of cells ``weakly below'' the section
$\sigma(z)$, i.e. the set $\{(i,j)\colon i\in[k],\, 0\le j\le z(i)\}$, called
the \emph{lower set} for $z$ (and for $\sigma(z)$).

Denote the number $k+b(1)+\cdots+ b(k)$ of cells in $D_{k,b}$ by
$N=N_{k,b}$. By a \emph{column-monotone filling} of $D$ we mean a bijective
map $t:D\to [N]$ such that for each $i\in[k]$, there holds
$i=t(i,0)<t(i,1)<\cdots< t(i,b(i))$. We refer to $T=(D,t)$ as a \emph{tableau}.
An example of tableau with $k=3$, $b(1)=4$ and $b(2)=b(3)=2$ is illustrated
in the picture.

\vspace{-0cm}
\begin{center}
\includegraphics[scale=0.7]{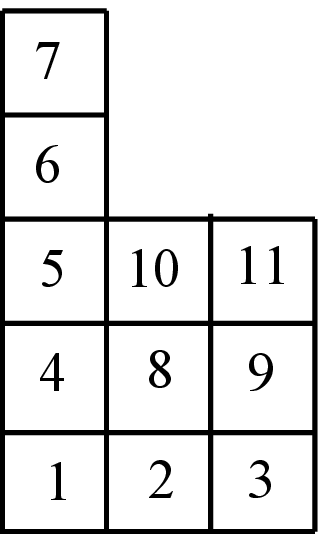}
\end{center}
\vspace{-0cm}

%We say that $(i,j)\in D$ is \emph{smaller} than $(i',j')\in D$ relative to the filling $t$ if $t(i,j)<t(i',j')$.
Given a tableau $T=(D_{k,b},t)$ and a \emph{quota} $q\in\Zset_{>0}$, we now
construct the choice function $C$ on $\Bscr=\Bscr_k,b$ by the following rule (due to V.~Danilov):
  \begin{numitem1} \label{eq:CF-T}
for $z\in\Bscr$, let $A$ consist of $k+\min(|z|,\, q)$ smaller elements of the lower
set $\Lambda_z=\Lambda_{\sigma(z)}$, relative to the filling
$t$; then $C(z)$ corresponds to the section of $D$ having $A$ as the lower set,
i.e. $\Lambda_{C(z)}=\Lambda_{\sigma(C(z))}=A$ (taking into account that
$(i,0)\in A$ for all $i\in [k]$).
  \end{numitem1}

It follows that $C(z)=z$ if $|z|\le q$, and $|C(z)|=q$ if $|z|>q$. Also the
column-monotonicity of $t$ implies that $\Lambda_{C(z)}\subseteq \Lambda_z$,
whence $C(z)\le z$.
 \begin{lemma} \label{lm:C-A1A2}
$C$ satisfies axioms (A1) and (A2).
  \end{lemma}
  \begin{proof}
To check (A1), consider $z,z'\in\Bscr$ with $z\ge z'\ge C(z)$ and assume that
$|z|>q$. Then $\Lambda_{z'}\supseteq \Lambda_{C(z)}$, both $\Lambda_{C(z)}$ and $\Lambda_{C(z')}$ have the same size $k+q$, and $t(i,j)>t(i',j')$ holds for all
$(i,j)\in \Lambda_{z'}-\Lambda_{C(z)}$ and $(i',j')\in \Lambda_{C(z)}$. It
follows that $\Lambda_{C(z')}=\Lambda_{C(z)}$, implying $C(z')=C(z)$, as
required.

To check (A2), consider $z,z'\in\Bscr$ with $z>z'$. One may assume that
$|z'|>q$. The desired inequality $C(z)\wedge z'\le C(z')$ is equivalent to the
inclusion $\Lambda_{C(z)}\cap \Lambda_{z'}\subseteq \Lambda_{C(z')}$. In turn,
the latter inclusion follows from the facts that $t(i,j)>t(i',j')$ for all
$(i,j)\in\Lambda_{z'}-\Lambda_{C(z)}$ and $(i',j')\in \Lambda_{C(z)}\cap
\Lambda_{z'}$ (since $(i,j)\in \Lambda_z-\Lambda_{C(z)}$ and $(i',j')\in\Lambda_{C(z)}$), and that
$|\Lambda_{z'}-\Lambda_{C(z)}|>|\Lambda_{C(z)}-\Lambda_{z'}|$ (in view of
$|z'|>q=|C(z)|$).
  \end{proof}

For our purposes, we apply the above construction to the following special
case:
  \begin{numitem1} \label{eq:k=3}
$k=3$, $q$ is even, $b(1)=q$, and $b(2)=b(3)=q/2=:p$,
  \end{numitem1}
and assign the following column-monotone filling on $D=D_{3,b}$:
  \begin{numitem1} \label{eq:t123}
$t(i,0):=i$ for $i=1,2,3$; ~$t(1,j):=3+j$ for $j=1,\ldots,q$;
~$t(2,j'):=2+q+2j'$ and $t(3,j'):=3+q+2j'$ for $j'=1,\ldots,p$.
 \end{numitem1}

The tableau of this sort with $q=4$ is illustrated in the picture above.

Next we consider a special sequence $z^0,z^1,\ldots,z^q$ in
$\Bscr=\Bscr_{3,b}$, where
  \begin{numitem1} \label{eq:z0-zq}
for $i$ even, $z^i(1):=i$ and $z^i(2)=z^i(3):=p-i/2$; and for $i$ odd,
$z^i(1):=i$, $z^i(2):=p-(i-1)/2$ and $z^i(3):=p-(i+1)/2$.
 \end{numitem1}

Then $|z^i|=q$ for $i=0,\ldots,q$. Also from~\refeq{t123} and~\refeq{z0-zq} one
can conclude that
  \begin{numitem1} \label{eq:special}
in the lower set $\Lambda_{z^i}$, the maximum value of $t$ is attained at the
upmost element of 3rd column (namely, at the cell $(3,p-i/2)$) when $i$ is
even, and at the upmost element of 2nd column (namely, at $(2,p-(i-1)/2)$) when
$i$ is odd.
 \end{numitem1}
(Indeed, computing the fillings in the cells of  section $\sigma(z^i)$ in
columns 2 and 3, we have: $t(2,p-i/2)=2+q+2p-i$ and $t(3,p-i/2)=3+q+2p-i$ when
$i$ is even; and $t(2,p-(i-1)/2)=2+q+2p-i+1=2q-i+3$ and
$t(3,p-(i+1)/2)=3+q+2p-i-1=2q-i+2$ when $i$ is odd.)

Property~\refeq{special} leads to the following important fact.
  \begin{corollary} \label{cor:1-interest}
{\rm(i)} The element 1 is interesting under $z^i$ for each $i<q$. More
precisely, $C(z^i+(1,0,0))=z^i+(1,0,-1)=z^{i+1}$ for $i$ even (and $i\ne q$),
and $C(z^i+(1,0,0))=z^i+(1,-1,0)=z^{i+1}$ for $i$ odd.

{\rm(ii)} $z^0\prec z^1\prec\cdots\prec z^q$, where $\prec$ is the preference relation determined by $C$.
%on the assignments $z\in\Bscr$ with $|z|=q$ determined by the choice functions $C$.
  \end{corollary}
(Indeed, the equalities in~(i) can be concluded
from~\refeq{z0-zq},~\refeq{special} and the rule~\refeq{CF-T}. And~(ii) follows
from~(i) (since $C(z^i\vee z^{i+1})=z^{i+1}$, in view of $z^i\vee
z^{i+1}=z^i+(1,0,0)$.)

By Corollary~\ref{cor:1-interest}(i), for $i$ even ($i$ odd), the pair $(1,3)$
(resp. $(1,2)$) forms the tandem for the assignment $z^i$. Thus, for the choice
function $C$ on $\Bscr_{3,b}$ determined by the filling $t$ as in~\refeq{t123},
condition~(C) from Sect.~\SEC{addit_prop} is violated for any triple
$z^i,z^{i+1},z^{i+2}$.

Now we use the obtained CF $C$ to construct an example for which the input
graph $G$ is ``small'' but the corresponding poset is ``large'' (depends on the
capacities $b$).

The graph $G=(V=W\sqcup F,E)$ is drawn in the left fragment of  the picture
below.

\vspace{-0.2cm}
\begin{center}
\includegraphics[scale=0.8]{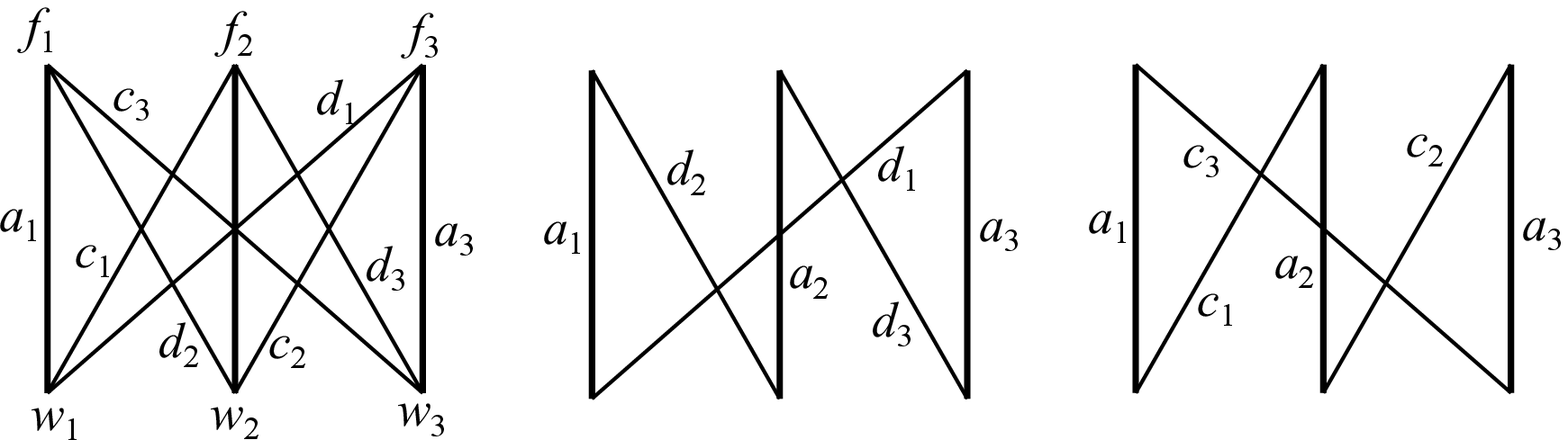}
\end{center}
\vspace{-0.3cm}

Here:
  \begin{numitem1}\label{eq:ex3x3}
\begin{itemize}
 \item[(i)]
$W=\{w_1,w_2,w_3\}$, ~$F=\{f_1,f_2,f_3\}$, and for $i=1,2,3$, the set $E_{w_i}$
consists of three edges $w_if_i$, $w_if_{i+1}$, $w_if_{i-1}$, denoted as
$a_i,c_i,d_i$, respectively (taking indexes modulo 3);
  \item[(ii)]
for each $w_i$, the preferences $>=>_{w_i}$ on $E_{w_i}$ are given by
$c_i>d_i>a_i$;
  \item[(iii)]
for each $f_i$, its incident edges $a_i,c_{i-1},d_{i+1}$ are identified
with copies of elements (column indexes) $1,2,3$ in the diagram $D=D_{3,b}$,
respectively;
  \item[(iv)]
for all $i=1,2,3$, the quotas $q(w_i)$ and $q(f_i)$ are the same even integer
$q>0$; the capacities $b(a_i),\,b(c_i),\,b(d_i)$ are as in~\refeq{k=3} for the
columns 1,2,3 in $D$ (respectively); and the choice function $C_{f_i}$ is a
copy of $C$ as in~\refeq{CF-T} for the tableau $(D,t)$ in question.
   \end{itemize}
   \end{numitem1}

Under these settings, we can construct g-allocations in $G$ as follows. For
each $i=0,1,\ldots,q$, consider the assignment $z^i$ as in~\refeq{z0-zq} and
define
  $$
  x^i(a_r):=z^i(1),\quad x^i(c_r):=z^i(2), \quad x^i(d_r):=z^i(3), \quad r=1,2,3.
  $$

In particular, $x^0(a_r)=0$ and $x^0(c_r)=x^0(d_r)=q/2$, $r=1,2,3$, which is
nothing else than the initial (the best for $W$ and the worst for $F$) stable
g-allocation $\xmin$ in the constructed model with $G,>,C,q$. Moreover, one can see
that: (a) each $x^i$ is stable as well; (b) the sequence
$x^0,x^1,\ldots,x^q=\xmax$ is the \emph{unique} full route $\Tscr$; (c) for $i$ even,
$x^{i+1}$ is obtained from $x_i$ by shifting with unit weight along the
rotation $L$ formed by the sequence of edges $a_1,d_2,a_2,d_3,a_3,d_1$ (where
$L^+=\{a_1,a_2,a_3\}$, illustrated in the middle fragment of the above picture;
and (d) for $i$ odd, $x^{i+1}$ is obtained from $x_i$ by shifting with unit
weight along the rotation $L'$ formed by the sequence of edges
$a_1,c_3,a_3,c_2,a_2,c_1$ (where $L'^+=\{a_1,a_2,a_3\}$), illustrated in the
right fragment of the above picture. (So the family $\hat\Pi(\Tscr)$ defined in Remark~1 (Sect.~\SEC{poset_rot}) consists of $q/2$ pairs $(L,1)$ and $q/2$ pairs $(L',1)$.)

Thus, the corresponding rotational poset is viewed as a chain of $q/2$ copies
of $L$ and $q/2$ copies of $L'$ in which the copies of $L$ and $L'$ alternate
(and do not commute). This chain has $q+1$ ideals, which correspond to the
stable allocations $x^0,\ldots,x^q$. The even parameter $q>0$ can be chosen
arbitrarily, and we can conclude with the following
  \begin{corollary} \label{cor:Cnot}
In model SGAM, if condition~(C) is not imposed, then there is a series of
instances with a fixed graph $G$ such that the posets representing the lattices of stable g-allocations have unboundedly growing sizes.
 \end{corollary}

\noindent\textbf{Remark 3.} 
We know that in a general case of SGAM the poset giving a representation of stable g-allocations (in terms of closed functions on the poset) may contain multiple rotations $L$, even multiple pairs $(L,\tau)$, as it may be so for non-excessive routes (see Remark~1 in Sect.~\SEC{poset_rot}). (An example of such a behavior, with the poset formed by two rotations repeated $\bmax/2$ times each, was demonstrated above.) This poset can be explicitly constructed as follows, assuming that $\xmin$ is available.

We first form a full route to obtain the family $\hat\Pi$ of corresponding pairs $(L,\tau)$ (which is invariant, see~\refeq{hatPi}). Let $\Dscr$ be the set of \emph{different} rotations in $\hat\Pi$, and for $L\in\Dscr$, let $\hat\Pi_L$ be the set of pairs containing $L$ in $\hat\Pi$. 

For each $L\in\Dscr$, we construct a non-excessive route $\Tscr_L$ by acting in spirit of the method in Sect.~\SSEC{graph_H}. The procedure consists of $|\hat\Pi_L|=:k_L$ stages. At 1st stage, starting with $\xmin$, we form, step by step, a non-excessive route by using merely rotations different from $L$. As a result, we obtain $x\in\Sscr$ such that the set $\Pi(x)$ of admissible rotations along with their maximal weights for $x$  consists of a single pair $(L,\tau)$ (for if $\Pi(x)$ contained two or more pairs from $\hat\Pi_L$, they should be merged in a non-excessive route). This $x$ is just the beginning g-allocation at 2nd stage, which is performed in a similar way. And so on until we reach $\xmax$, obtaining $\Tscr_L$. Then we number the pairs in $\hat\Pi_L$ as $\pi^{L^1},\ldots,\pi^{L^{k_L}}$ according to their occurrences in $\Tscr_L$; clearly $\pi^{L^1}\lessdot\cdots\lessdot\pi^{L^{k_L}}$.

Now to establish the immediately preceding relations between pairs with different $L,L'\in\Dscr$, we again examine the obtained routes $\Tscr_L$. More precisely, for $\pi^{L^i}$, let $x$ be the g-allocation in the beginning of $(i+1)$-th stage of constructing $\Tscr_L$. Then there are exactly $|\Pi(x)|$ pairs in $\hat\Pi$ that immediately succeed $\pi^{L^i}$, and one can see that these are the pairs $\pi_j^{L'}$ such that $L'\in\Lscr(x)$, and $j$ is defined so that: the subroute of $\Tscr_L$ from $x$ to $\xmax$ uses the pairs $\pi_j^{L'},\ldots,\pi_{k_L'}^{L'}$ from $\hat\Pi_{L'}$ and only them.

Computing such pairs $(\pi_i^L,\pi^{L'}_j)$ for all $L\in\Dscr$, we obtain the generating graph (Hasse diagram) of the required poset. Since $|\hat\Pi|<\bmax|E|^2$ (see Remark~1), the running time of the above method can be roughly estimated as $(\bmax)^2$ times a polynomial in $|E|$.

\end{document}